\documentclass{article}

\usepackage{graphicx} 
\usepackage{mathtools}
\mathtoolsset{showonlyrefs=true}
\usepackage{bm}
\usepackage{amssymb}
\usepackage{amsthm}
\usepackage{amsmath}
\usepackage{todonotes}
\usepackage[left=30mm,right=30mm,top=30mm,bottom=30mm]{geometry}
\usepackage{cases}
\usepackage{stmaryrd}

\newcommand{\relmiddle}[1]{\mathrel{}\middle#1\mathrel{}}
\newcommand{\Div}{\operatorname{div}}

\newcommand{\dist}{\operatorname{dist}}

\theoremstyle{plain}
\newtheorem{theorem}{Theorem}[section]
\newtheorem{lemma}[theorem]{Lemma}
\newtheorem{proposition}[theorem]{Proposition}
\theoremstyle{definition}

\theoremstyle{remark}
\newtheorem{remark}[theorem]{Remark}

\numberwithin{equation}{section}

\allowdisplaybreaks

\title{On the existence of a singular limit equation for a model of a self-propelled object motion}
\makeatletter
\renewcommand\@date{{%
  \vspace{-\baselineskip}%
  \large\centering
  \begin{tabular}{@{}c@{}}
    Masaharu Nagayama\textsuperscript{1} \\
    \normalsize nagayama@es.hokudai.ac.jp
  \end{tabular}%
  \quad\quad
  \begin{tabular}{@{}c@{}}
    Koya Sakakibara\textsuperscript{2,3} \\
    \normalsize ksakaki@se.kanazawa-u.ac.jp
  \end{tabular}%
  \quad\quad
  \begin{tabular}{@{}c@{}}
    Keisuke Takasao\textsuperscript{4} \\
    \normalsize k.takasao@math.kyoto-u.ac.jp
  \end{tabular}

  \bigskip

  \textsuperscript{1}Research Center of Mathematics for Social Creativity, Research Institute for Electronic Science, Hokkaido University\par
  \textsuperscript{2}Faculty of Mathematics and Physics, Institute of Science and Engineering, Kanazawa University\par
  \textsuperscript{3}RIKEN iTHEMS\par
  \textsuperscript{4}Department of Mathematics, Graduate School of Science, Kyoto University
}}

\begin{document}

\maketitle

\begin{abstract}
    In this paper, a phase‑field model is introduced to describe the evolution of a deformable, self‑propelled object driven by surface‑tension effects. 
    The model couples an Allen--Cahn‑type equation, which distinguishes the body from the surrounding fluid, with a reaction-diffusion equation for the surfactant concentration. 
    As the interface‑thickness parameter $\varepsilon$ tends to zero, it is shown that the phase‑field model converges to a sharp-interface limit coupled with a reaction-diffusion equation.
    In particular, the normal velocity is given by the mean curvature, surface tension, and volume-preserving effect.
\end{abstract}

\section{Introduction}
A self-propelled system is a system in motion by the action of an object deforming or an object changing its field of motion. 
The motion of bacteria and cells in living systems and the non-living systems, such as camphor and pentanol droplets, are considered to belong to self-propelled systems because they have internal mechanisms to generate motion.
In recent years, mathematical models have been used to analyze keratocyte motion~\cite{SRL2010}, cell collective motion~\cite{GG1992, HTN2004, N2012} and cell division~\cite{AN2019}, while non-living systems such as camphor disks motion~\cite{MNDH2004,PhysRevE.92.022910}, the motion of pentanol droplet~\cite{NSKY2005}, the regular motion of oil droplet~\cite{SMHY2005}, and various other motions have been analyzed through mathematical models.  
In this study, we focus on self-propelled systems, where the motion mechanism is primarily governed by surface tension, which is the motion of non-living systems such as camphor disks and pentanol droplets.

Since the formulation of the motion model without shape deformation has already been done and some comparisons between experimental and mathematical model results and mathematical analyses have been performed \cite{IKN2014, KIN2013}, the first author and second author have therefore developed the following mathematical model that can handle self-propelled object motion with shape change, which consists of the Allen-Cahn equation~\cite{AC1979} and the parabolic equation for surfactant molecular concentrations (see \cite{nagayama2023reaction} for more details):
\begin{equation}
    \begin{dcases*}
        \varepsilon^2\tau\frac{\partial\varphi^\varepsilon}{\partial t}=\varepsilon^2\sigma^2\triangle\varphi^\varepsilon+\varphi^\varepsilon(1-\varphi^\varepsilon)\left(\varphi^\varepsilon-a(S[\varphi](t),u^\varepsilon,\varepsilon)\right)&in $\Omega\times(0,T)$,\\
        \frac{\partial u^\varepsilon}{\partial t}=\triangle u^\varepsilon-ku^\varepsilon+\varphi^\varepsilon&in $\Omega\times(0,T)$,\\
        \varphi^\varepsilon|_{t=0}=\varphi_0^\varepsilon&in $\Omega$,\\
        u^\varepsilon|_{t=0}=u_0^\varepsilon&in $\Omega$,
    \end{dcases*}
    \label{eq:old_problem}
\end{equation}
where $\Omega=(\mathbb{R}/\mathbb{Z})^n$ with $n=2,3$, $\varphi_0^\varepsilon$ and $u_0^\varepsilon$ are initial conditions such that $0\le\varphi_0^\varepsilon\le1$ and $u_0^\varepsilon\ge0$,
\begin{align}
    a(S[\varphi^\varepsilon](t),u^\varepsilon,\varepsilon)
    &=\frac{1}{2}+\varepsilon\left(-\gamma(u)+S[\varphi^\varepsilon](t)\right),\\
    S[\varphi^\varepsilon](t)&=\alpha\left(\int_\Omega\varphi^\varepsilon\,\mathrm{d}x-\int_\Omega\varphi_0^\varepsilon\,\mathrm{d}x\right),
\end{align}
and $\gamma\in C^2((0,\infty))$ is a smooth monotonically decreasing function, for instance,
\begin{equation}
    \gamma(u)=\frac{1}{1+(u/u_1)^m}+\gamma_0,\quad
    u>0,\quad m \in \mathbb{N}.
\end{equation}
Here $\varphi^\varepsilon(x, t)$ denotes $0 \le \varphi^\varepsilon(x, t) < 1/2$ as the water surface and $1/2 \le \varphi^\varepsilon(x, t) \le 1$ as a self-propelled object.
The unknown function $u(x, t)$ represents the concentration of the surfactant molecules on the water surface.
The function $\gamma(u)$ is the surface tension, where $\gamma_0$ is the surface tension at the critical micelle concentration of surfactant, $\gamma_0 + \gamma_1$ is the surface tension of the water surface and $s$ is a positive constant.\\
The formulation of the mathematical model above is a modeling technique called the phase field method, which has been used in the crystal growth model in undercooled liquids~\cite{K1993} and the crystal interface kinetic model~\cite{OKS2001}.
In recent years, the phase field method has been used in the formulation of cellular motion living systems using the Allen-Cahn equations~\cite{AN2019,N2012} and the Cahn-Hilliard equation\cite{SRL2010}, which often appears as a representation model in the field of materials science\cite{AHS2021}.\\
Assuming that $\gamma$ depends only on the space variable $x$ in model \eqref{eq:old_problem}, we obtain the following quasi-volume conserving Allen-Cahn equation:
\begin{equation} \label{eq01-03}
\varepsilon^{2} \tau \frac{\partial \varphi^\varepsilon}{\partial t} = \varepsilon^2 \sigma^2 \Delta \varphi^\varepsilon + \varphi^\varepsilon(1 - \varphi^\varepsilon)(\varphi^\varepsilon - 1/2 + \varepsilon ( \gamma(x) + S[\varphi^\varepsilon](t)) ).
\end{equation}
Unfortunately, \eqref{eq01-03} can not be derived as the $L^2$ gradient flow under the following condition:
\begin{equation}
   S[\varphi^\varepsilon](t)=\alpha\left(\int_\Omega\varphi^\varepsilon\,\mathrm{d}x-\int_\Omega\varphi_0^\varepsilon\,\mathrm{d}x\right).
\end{equation}
In this paper, we change the term $S[\varphi^{\varepsilon}](t)$ as follows:
\begin{equation}
    \tilde{S}[\varphi^{\varepsilon}](t)
    \coloneqq\alpha\left(
        \int_\Omega G(\varphi^{\varepsilon})\,\mathrm{d}x-\int_\Omega G(\varphi^{\varepsilon}_0)\,\mathrm{d}x
    \right),\qquad
    G(\varphi)
    \coloneqq\frac{1}{2}\varphi^2-\frac{1}{3}\varphi^3.
\end{equation}
Then equation\eqref{eq01-03} with $S$ replaced by $\tilde{S}$ can be derived as the $L^2$ gradient flow with the following energy $E(t)$:
\begin{equation}\label{eq01-5}
\begin{aligned}
E(t) = & \int_{\Omega}\left( \frac{\varepsilon \sigma^2 |\nabla \varphi^{\varepsilon} |^2}{2} + \frac{W(\varphi^{\varepsilon})}{2 \varepsilon} \right) \mathrm{d}x \\
& + \frac{\alpha}{2}\left( \int_{\Omega} G(\varphi^{\varepsilon}(x, t)) \mathrm{d}x - \int_{\Omega} G(\varphi^{\varepsilon}(x, 0))  \mathrm{d}x \right)^2 + \int_{\Omega} \gamma(x)(1 - G(\varphi^{\varepsilon}(x, t)) ) \mathrm{d}x, 
\end{aligned} 
\end{equation}
where $W$ is potential 
\begin{equation}
    W(\varphi)
    \coloneqq\frac{\varphi^2(1-\varphi)^2}{2}.
\end{equation}
We see that
\begin{equation}
    \varphi^\varepsilon(1-\varphi^\varepsilon)(\varphi^\varepsilon-a(\tilde{S}[\varphi^\varepsilon](t),u^\varepsilon,\varepsilon))
    =-\frac{W'(\varphi^\varepsilon)}{2}-\varepsilon G'(\varphi^\varepsilon)(-\gamma(u^\varepsilon)+\tilde{S}[\varphi^\varepsilon](t)).
\end{equation}
Namely, our mathematical model in this paper can be given by
\begin{equation}
    \begin{dcases*}
        \varepsilon^2\tau\frac{\partial\varphi^\varepsilon}{\partial t}=\varepsilon^2\sigma^2\triangle\varphi^\varepsilon-\frac{W'(\varphi^\varepsilon)}{2}-\varepsilon G'(\varphi^\varepsilon)\left(-\gamma(u^\varepsilon)+\tilde{S}[\varphi^\varepsilon](t)\right)&in $\Omega\times(0,T)$,\\
        \frac{\partial u^\varepsilon}{\partial t}=\triangle u^\varepsilon-ku^\varepsilon+\varphi^\varepsilon&in $\Omega\times(0,T)$,\\
        \varphi^\varepsilon|_{t=0}=\varphi_0^\varepsilon&in $\Omega$,\\
        u^\varepsilon|_{t=0}=u_0^\varepsilon&in $\Omega$.
    \end{dcases*}
    \label{eq:target_problem}
\end{equation}
The dynamics of the above model \eqref{eq:target_problem} is qualitatively the same as the old model \eqref{eq:old_problem}, which will be described in a forthcoming paper.
Therefore, we re-propose the mathematical model \eqref{eq:target_problem} as a mathematical model for the self-propelled object motion involving deformation, and we treat the model equation \eqref{eq:target_problem} as an analytical object in this paper.
It is well known that the singular limit of the solution to the standard Allen-Cahn equation corresponds to the mean curvature flow. Similarly, we can obtain some geometric evolution equations for our model (see \cite{nagayama2023reaction} and Theorem~\ref{thm:main} below).
From another perspective, it can be said that the phase field methods can be used as an approximation of the solution of geometric equations. 
The goal of this paper is not only to characterize the singular limits of equations \eqref{eq:target_problem}, but also to prove the existence of weak solutions to the system, including the geometric equations.

Note that $\varphi^\varepsilon$ converges to some characteristic function $\varphi$ as $\varepsilon\searrow0$ under suitable assumptions (see \cite[Proposition 8.3]{takasao2016existence}).
Then, we formally obtain that
\begin{equation}
    \int_\Omega G(\varphi^\varepsilon)\,\mathrm{d}x
    \approx
    \int_{\Omega\cap(\{\varphi=1\})}G(\varphi)\,\mathrm{d}x
    +\int_{\Omega\cap\{\varphi=0\}}G(\varphi)\,\mathrm{d}x
    =\frac{1}{6}|\Omega\cap\{x\in\Omega\mid\varphi(x)=1\}|.
\end{equation}
For the sake of simplicity in the following discussion, we set $\tau=\sigma=1$.
Let $(\varphi^\varepsilon,u^\varepsilon)$ be a classical solution to \eqref{eq:target_problem}.
Throughout the following discussion, we suppose that
\begin{equation}
    \sup_{\varepsilon>0}\|u_0^\varepsilon\|_{H^1(\Omega)}<\infty,\quad
    \sup_{\varepsilon>0}\left[
        \int_\Omega\left(
            \frac{\varepsilon|\nabla\varphi_0^\varepsilon|^2}{2}+\frac{W(\varphi_0^\varepsilon)}{2\varepsilon}
        \right)\,\mathrm{d}x
    \right]
    <\infty,\quad
    0<\varphi_0^\varepsilon(x)<1\quad(x\in\Omega)
    \label{eq:u0-phi0-bound}
\end{equation}
holds.
The purpose of this research is to show the existence of singular limit equations for \eqref{eq:target_problem}, and we have been able to show the existence of singular limit equations in a weak sense as follows:
\begin{theorem}
    \label{thm:main}
    Let $\mu_t ^\varepsilon$ be a Radon measure defined in Theorem \ref{thm:GM} below.
    Then there exists a positive sequence $\{\varepsilon_i\}_{i=1}^\infty$ 
    with $\varepsilon _i \to 0$ such that the limit $\varphi\coloneqq\lim_{i\to\infty}\varphi^{\varepsilon_i}$ exists for a.e. $(x,t)\in\Omega\times(0,\infty)$ and the limit measure $\mu_t =\lim_{i\to \infty} \mu_t ^{\varepsilon _i}$ exists for any $t\geq 0$, where $\mu_t$ satisfies Theorem \ref{thm:GM} (II), and the generalized velocity vector $\vec{v}$ with respect to $\mu_t$ is given by
    \begin{equation}
        \vec{v}
    =\vec{h}-\frac{1}{6}\tilde{S}[\varphi]\left(\frac{\mathrm{d}\mathcal{H}^{n-1}|_{\partial^\ast\Omega_{\mathrm{in}}(\varphi,t)}}{\mathrm{d}\mu_t}\right)\vec{\nu}+\frac{1}{6}\gamma(u)\left(
        \frac{\mathrm{d}\mathcal{H}^{n-1}|_{\partial^\ast\Omega_{\mathrm{in}}(\varphi,t)}}{\mathrm{d}\mu_t}
    \right)\vec{\nu}.
    \end{equation}
    Namely, 
\begin{multline}
    \int_0^T\int_\Omega\vec{v}\cdot\vec{\eta}\,\mathrm{d}\mu_t\,\mathrm{d}t
    =\int_0^T\int_\Omega\vec{h}\cdot\vec{\eta}\,\mathrm{d}\mu_t\,\mathrm{d}t\\
    +\int_0^T\int_\Omega\left[-\frac{1}{6}\tilde{S}[\varphi]\left(\frac{\mathrm{d}\mathcal{H}^{n-1}|_{\partial^\ast\Omega_{\mathrm{in}}(\varphi,t)}}{\mathrm{d}\mu_t}\right)\vec{\nu}+\frac{1}{6}\gamma(u)\left(
        \frac{\mathrm{d}\mathcal{H}^{n-1}|_{\partial^\ast\Omega_{\mathrm{in}}(\varphi,t)}}{\mathrm{d}\mu_t}
    \right)\vec{\nu}\right]\cdot\vec{\eta}\,\mathrm{d}\mu_t\,\mathrm{d}t
\end{multline}
holds for any $\vec{\eta}\in(C_0^1(\Omega\times(0,T)))^n$,
where $\vec{h}$ and $\vec{\nu}$ are the generalized mean curvature vector and the outer unit normal vector, respectively.
\end{theorem}

\begin{remark}
    Formally, $\mu_t$ corresponds to the Hausdorff measure with respect to the boundary of $\{x\in\Omega\mid\lim_{\varepsilon\to0}\varphi^{\varepsilon_i}(x,t)=1\}$.
    In addition, $\vec{v}$ corresponds to the normal velocity vector of the boundary.
\end{remark}

Throughout this paper, we often denote $\varepsilon_i$
as $\varepsilon$ for simplicity.
This paper is organized as follows.
In Chapter 2, we give the basic evaluations necessary for the proof, and in Chapter 3, we prove the main theorem. Finally, we give a summary.

\section{Preliminaries}

We prepare several notions used in this paper.
Set
\begin{equation}
    M_\gamma\coloneqq\sup_{u\ge0}\gamma(u).
\end{equation}
Define the region $\Omega_{\mathrm{in}}(\varphi,t)$ as
\begin{equation}
    \Omega_{\mathrm{in}}(\varphi,t)\coloneqq\left\{x\in\Omega\relmiddle|\varphi(x,t)>\frac{1}{2}\right\}.
\end{equation}
We often denote the above set by $\Omega_{\mathrm{in}}(t)$ unless there is a confusion.
Define the energy $E$ as
\begin{equation}
    E(t)\coloneqq E_s(t)+E_p(t),\label{eq:energy}
\end{equation}
where
\begin{equation}
    E_s(t)\coloneqq\int_\Omega\left(\frac{\varepsilon|\nabla\varphi^\varepsilon|^2}{2}+\frac{W(\varphi^\varepsilon)}{2\varepsilon}\right)\,\mathrm{d}x,\qquad
    E_p(t)\coloneqq\frac{\alpha}{2}\left(\int_\Omega G(\varphi^\varepsilon)\,\mathrm{d}x-\int_\Omega G(\varphi_0^\varepsilon)\,\mathrm{d}x\right)^2,
\end{equation}

\begin{theorem}
\label{thm:MP}
Assume that $0< \varphi_0 ^\varepsilon (x)<1$ and $u_0^\varepsilon(x)>0$ hold for any $x \in \Omega$. 
Then, we have $0< \varphi ^\varepsilon (x,t)<1$ and $u^\varepsilon(x,t)>0$ for any $x \in \Omega$ and $t \in [0,\infty)$.
\end{theorem}

\begin{proof}
Let $\overline{\varphi} (t) = 1- D_1 e^{-D_2 t}$,
where 
\begin{equation}
    D_1= 1 - \max \left\{ \frac12 , \max_{x \in \Omega} \varphi_0 ^\varepsilon(x) \right\},\qquad
    D_2= \varepsilon^{-2} \left[\frac12 + \varepsilon \left(M_\gamma + \frac{1}{3} |\Omega| \alpha\right)\right].
\end{equation}
We may assume that there exists $T>0$ such that
$0< \varphi ^\varepsilon (x,t)<1$ for any $x \in \Omega$ and $t \in [0,T)$.
Note that $1-D_1\ge\max_{x\in\Omega}\varphi_0^\varepsilon(x)$ holds; that is, $\varphi_0^\varepsilon\le\overline{\varphi}(0)$.
Since the function $[0,1]\ni\varphi\mapsto G(\varphi)$ is monotonically increasing and $0\le G(\varphi)\le1/6$ holds for any $\varphi\in[0,1]$, we obtain
\begin{equation}
    |\tilde{S}[\varphi^\varepsilon](t)|
    \le\alpha\left(\int_\Omega G(\varphi^\varepsilon)\,\mathrm{d}x+\int_\Omega G(\varphi_0^\varepsilon)\,\mathrm{d}x\right)
    \le\frac1{3}|\Omega|\alpha,
    \qquad t \in [0,T).
    \label{eq:St_estimate}
\end{equation}
Therefore, for any $t \in(0,T)$, we have
\begin{align*}
&\varepsilon^2 \triangle\overline{\varphi}-\frac{W' (\overline{\varphi})}{2}-\varepsilon G'(\overline{\varphi})\left(-\gamma(u^\varepsilon)+\tilde{S}[\varphi^\varepsilon](t)\right) \\
&\leq
\overline{\varphi} (1-\overline{\varphi}) \left(\overline{\varphi} -\frac12\right)
+ \varepsilon (1- \overline{\varphi}) \overline{\varphi} 
\left(M_\gamma + \frac1{3} |\Omega| \alpha\right) \\
&\leq
\frac12 (1-\overline{\varphi}) 
+ \varepsilon (1- \overline{\varphi})  
\left(M_\gamma + \frac1{3} |\Omega| \alpha\right)\\
&= D_1 e^{-D_2 t} \left[\frac12 +\varepsilon \left(M_\gamma + \frac1{3} |\Omega| \alpha\right) \right]\\
&=\varepsilon^2\overline{\varphi}_t.
\end{align*}
Therefore, $\overline{\varphi}$ is a super solution to 
\begin{equation}\label{equationforMP}
\varepsilon^2\frac{\partial\phi}{\partial t}=\varepsilon^2 \triangle\phi-\frac{W'(\phi)}{2}-\varepsilon G'(\phi)\left(-\gamma(u^\varepsilon)+\tilde{S}[\varphi^\varepsilon](t)\right).
\end{equation}
Note that at the point $(x,t)$ where $\varphi ^\varepsilon =\phi$ holds, the second and third terms on the right-hand side of \eqref{equationforMP} coincide with those of \eqref{eq:target_problem}, and thus the maximum principle can be used.
Similarly, we find that $\underline{\varphi}(t)=D_3\mathrm{e}^{-D_2t}$ with $D_3=\min_{x\in\Omega}\varphi_0^\varepsilon(x)$ is a sub solution to the first equation in \eqref{equationforMP}.
Hence, the maximum principle implies 
\begin{equation}
0 < \underline{\varphi}(t) \leq \varphi^\varepsilon (x,t)
\leq \overline{\varphi} (t) <1,\qquad x\in\Omega,\quad t \in [0,T).    
\label{maximumprinciple}
\end{equation}

Concerning $u^\varepsilon$, let us consider a function $\underline{u}(t)=D_4\mathrm{e}^{-kt}$ with $D_4=\min_{x\in\Omega}u_0^\varepsilon(x)$.
Then, we have
\begin{equation}
    \triangle\underline{u}-k\underline{u}+\varphi^\varepsilon
    \ge-k\underline{u}
    =\underline{u}_t.
\end{equation}
Therefore, $\underline{u}$ is a sub solution to the second equation in \eqref{eq:target_problem}, and the maximum principle implies
\begin{equation}
    0<\underline{u}(t)\le u^\varepsilon(x,t),\qquad x\in\Omega,\quad t\in[0,T).
    \label{eq:u_maximum-principle}
\end{equation}
One can easily check that \eqref{maximumprinciple} and \eqref{eq:u_maximum-principle} hold for any $T>0$.
\end{proof}

\begin{lemma}
    \label{lem:u_converge}
    There exists a subsequence of $\{u^\varepsilon\}_{\varepsilon>0}$, denoted by the same symbol, and $u\in H^1(\Omega\times(0,T))$ such that $\{u^\varepsilon\}_{\varepsilon>0}$ weakly converges to $u$ in $H^1(\Omega\times(0,T))$ and strongly in $L^2(\Omega\times(0,T))$.
\end{lemma}

\begin{proof}
    We first show that $\{u^\varepsilon\}_{\varepsilon>0}$ is uniformly bounded in $H^1(\Omega\times(0,T))$; that is,
    \begin{equation}
        \sup_{\varepsilon>0}\|u^\varepsilon\|_{H^1(\Omega\times(0,T))}<\infty.\label{eq:estimate_u-H1}
    \end{equation}
    By using \eqref{eq:target_problem} and performing integration by parts, we have
    \begin{align*}
        \frac{1}{2}\frac{\mathrm{d}}{\mathrm{d}t}\int_\Omega(u^\varepsilon)^2\,\mathrm{d}x
        &=\int_\Omega u^\varepsilon u_t^\varepsilon\,\mathrm{d}x
        =\int_\Omega u^\varepsilon\left(\triangle u^\varepsilon-ku^\varepsilon+\varphi^\varepsilon\right)\,\mathrm{d}x
        =\int_\Omega\left(-|\nabla u^\varepsilon|^2-k(u^\varepsilon)^2+u^\varepsilon\varphi^\varepsilon\right)\,\mathrm{d}x.
    \end{align*}
    Then, applying the Young's inequality and using Theorem~\ref{thm:MP}, we obtain
    \begin{align*}
        \int_\Omega\left(-|\nabla u^\varepsilon|^2-k(u^\varepsilon)^2+u^\varepsilon\varphi^\varepsilon\right)\,\mathrm{d}x
        &\le\int_\Omega\left(-|\nabla u^\varepsilon|^2-k(u^\varepsilon)^2+\frac{k}{2}(u^\varepsilon)^2+\frac{1}{2k}(\varphi^\varepsilon)^2\right)\,\mathrm{d}x\\
        &\le-\int_\Omega\left(|\nabla u^\varepsilon|^2+\frac{k}{2}(u^\varepsilon)^2\right)\,\mathrm{d}x+\frac{1}{2k}|\Omega|;
    \end{align*}
    that is,
    \begin{equation}
        \frac{1}{2}\frac{\mathrm{d}}{\mathrm{d}t}\int_\Omega(u^\varepsilon)^2\,\mathrm{d}x
        \le-\int_\Omega\left(|\nabla u^\varepsilon|^2+\frac{k}{2}(u^\varepsilon)^2\right)\,\mathrm{d}x+\frac{1}{2k}|\Omega|.
    \end{equation}
    Therefore, integrating the above inequality with respect to $t$ from $t=0$ to $t=T$ implies that
    \begin{equation}
        \left.\frac{1}{2}\int_\Omega(u^\varepsilon)^2\,\mathrm{d}x\right|_{t=T}+\int_0^T\int_\Omega\left(|\nabla u^\varepsilon|^2+\frac{k}{2}(u^\varepsilon)^2\right)\,\mathrm{d}x\,\mathrm{d}t
        \le\left.\frac{1}{2}\int_\Omega(u^\varepsilon)^2\,\mathrm{d}x\right|_{t=0}+\frac{T}{2k}|\Omega|.
    \end{equation}
    In particular, the above inequality yields that
    \begin{equation}
        \int_0^T\int_\Omega\frac{k}{2}(u^\varepsilon)^2\,\mathrm{d}x,\ \int_0^T\int_\Omega|\nabla u^\varepsilon|^2\,\mathrm{d}x\,\mathrm{d}t\le\left.\frac{1}{2}\int_\Omega(u^\varepsilon)^2\,\mathrm{d}x\right|_{t=0}+\frac{T}{2k}|\Omega|.
    \end{equation}
    Since $\{u_0^\varepsilon\}_{\varepsilon>0}$ is uniformly bounded in $H^1(\Omega)$ by the assumption~\eqref{eq:u0-phi0-bound}, the above estimate yields that
    \begin{equation}
        \sup_{\varepsilon>0}\|u^\varepsilon\|_{L^2(\Omega\times(0,T))}<\infty,\qquad
        \sup_{\varepsilon>0}\|\nabla u^\varepsilon\|_{L^2(\Omega\times(0,T))}<\infty.
        \label{eq:u_nabla-u_bound}
    \end{equation}
    Similarly, using \eqref{eq:target_problem} and performing integration by parts, we have
    \begin{align*}
        \frac{1}{2}\frac{\mathrm{d}}{\mathrm{d}t}\int_\Omega|\nabla u^\varepsilon|^2\,\mathrm{d}x
        &=\int_\Omega\nabla u^\varepsilon\cdot\nabla u_t^\varepsilon\,\mathrm{d}x
        =-\int_\Omega(\triangle u^\varepsilon) u_t^\varepsilon\,\mathrm{d}x
        =-\int_\Omega(u_t^\varepsilon+ku^\varepsilon-\varphi^\varepsilon)u_t^\varepsilon\,\mathrm{d}x\\
        &=-\int_\Omega(u_t^\varepsilon)^2\,\mathrm{d}x-\int_\Omega ku^\varepsilon u_t^\varepsilon\,\mathrm{d}x+\int_\Omega\varphi^\varepsilon u_t^\varepsilon\,\mathrm{d}x.
    \end{align*}
    Using the Young's inequality and Theorem~\ref{thm:MP}, we obtain
    \begin{align}
        &-\int_\Omega(u_t^\varepsilon)^2\,\mathrm{d}x-\int_\Omega ku^\varepsilon u_t^\varepsilon\,\mathrm{d}x+\int_\Omega\varphi^\varepsilon u_t^\varepsilon\,\mathrm{d}x\\
        &\le-\int_\Omega(u_t^\varepsilon)^2\,\mathrm{d}x+\int_\Omega\left(\frac{1}{4}(u_t^\varepsilon)^2+k^2(u^\varepsilon)^2\right)\,\mathrm{d}x+\int_\Omega\left(\frac{1}{4}(u_t^\varepsilon)^2+(\varphi^\varepsilon)^2\right)\,\mathrm{d}x\\
        &\le-\frac{1}{2}\int_\Omega(u_t^\varepsilon)^2\,\mathrm{d}x+k^2\int_\Omega(u^\varepsilon)^2\,\mathrm{d}x+|\Omega|;
    \end{align}
    that is,
    \begin{equation}
        \frac{1}{2}\frac{\mathrm{d}}{\mathrm{d}t}\int_\Omega|\nabla u^\varepsilon|^2\,\mathrm{d}x
        \le-\frac{1}{2}\int_\Omega(u_t^\varepsilon)^2\,\mathrm{d}x+k^2\int_\Omega(u^\varepsilon)^2\,\mathrm{d}x+|\Omega|.
    \end{equation}
    Therefore, integrating the above inequality with respect to $t$ from $t=0$ to $t=T$ yields that
    \begin{equation}
        \left.\frac{1}{2}\int_\Omega|\nabla u^\varepsilon|^2\,\mathrm{d}x\right|_{t=T}+\frac{1}{2}\int_0^T\int_\Omega(u_t^\varepsilon)^2\,\mathrm{d}x\,\mathrm{d}t
        \le\left.\frac{1}{2}\int_\Omega|\nabla u^\varepsilon|^2\,\mathrm{d}x\right|_{t=0}+k^2\int_0^T\int_\Omega(u^\varepsilon)^2\,\mathrm{d}x+T|\Omega|.
    \end{equation}
    In particular, the above inequality leads to
    \begin{equation}
        \frac{1}{2}\int_0^T\int_\Omega(u_t^\varepsilon)^2\,\mathrm{d}x\,\mathrm{d}t
        \le\left.\frac{1}{2}\int_\Omega|\nabla u^\varepsilon|^2\,\mathrm{d}x\right|_{t=0}+k^2\int_0^T\int_\Omega(u^\varepsilon)^2\,\mathrm{d}x+T|\Omega|.
    \end{equation}
    Note that $\{u_0^\varepsilon\}_{\varepsilon>0}$ is uniformly bounded in $H^1(\Omega)$ by the assumption~\eqref{eq:u0-phi0-bound} and so is $\{u^\varepsilon\}_{\varepsilon>0}$ by \eqref{eq:u_nabla-u_bound}.
    Since $\{u^\varepsilon\}_{\varepsilon>0}$ is uniformly bounded in $L^2(\Omega\times(0,T))$ in $H^1(\Omega)$ by \eqref{eq:u_nabla-u_bound}, the above inequality implies that
    \begin{equation}
        \sup_{\varepsilon>0}\|u_t^\varepsilon\|_{L^2(\Omega\times(0,T))}<\infty.
    \end{equation}
    Thus, we have shown the desired estimate~\eqref{eq:estimate_u-H1}.

    Since $H^1(\Omega\times(0,T))$ is weakly compact, the above estimate shows that there exists a subsequence of $\{u^\varepsilon\}_{\varepsilon>0}$, denoted by the same symbol, and a function $u\in H^1(\Omega\times(0,T))$ such that $\{u^\varepsilon\}_{\varepsilon>0}$ weakly converges to $u$ in $H^1(\Omega\times(0,T))$.
    The Rellich--Kondrachov theorem implies that $H^1(\Omega\times(0,T))$ is compactly embedded into $L^2(\Omega\times(0,T))$.
    Therefore, taking a subsequence of $\{u^\varepsilon\}_{\varepsilon>0}$ if necessary, denoted by the same symbol, $\{u^\varepsilon\}_{\varepsilon>0}$ strongly converges to $u$ in $L^2(\Omega\times(0,T))$.

\end{proof}

\begin{lemma}
    \label{lem:estimate}
    \begin{enumerate}
        \item There is a constant $C(T)$ such that the estimate
        \begin{equation}
            E(T)+\int_0^T\int_\Omega\varepsilon\left(\varphi_t^\varepsilon\right)^2\,\mathrm{d}x\,\mathrm{d}t\le C(T)
        \end{equation}
        holds for any $\varepsilon\in(0,1)$.
        More precisely, $C(T)$ is explicitly given by
        \begin{equation}
            C(T)=2\mathrm{e}^{2M_\gamma^2T}E(0).
            \label{eq:CT}
        \end{equation}
        \item It holds that
        \begin{equation}
            \lim_{\alpha\to\infty}\int_\Omega G(\varphi^\varepsilon)\,\mathrm{d}x
            =\int_\Omega G(\varphi_0^\varepsilon)\,\mathrm{d}x
        \end{equation}
        for each $t\in(0,\infty)$, and
        \begin{align}
             &\sup_{\substack{t\in(0,T),\\\varepsilon>0}}E_s(t)<\infty,\\
            &\sup_{\varepsilon>0}\int_0^T\int_\Omega\frac{1}{\varepsilon}\left[
                G'(\varphi^\varepsilon)\left(
                    \gamma(u^\varepsilon)-\tilde{S}[\varphi^\varepsilon](t)
                \right)
            \right]^2\,\mathrm{d}x\,\mathrm{d}t
            <\infty.
        \end{align}
        In particular, it holds that
        \begin{equation}
            \sup_{\varepsilon\in(0,1)}\int_0^T\int_\Omega\left(\varepsilon\triangle\varphi^\varepsilon-\frac{W'(\varphi^\varepsilon)}{2\varepsilon}\right)^2\,\mathrm{d}x\,\mathrm{d}t<\infty.
        \end{equation}
    \end{enumerate}
\end{lemma}

\begin{proof}
(I)
Performing integration by parts, we have
\begin{equation}
    \frac{\mathrm{d}}{\mathrm{d}t}E_s(t)=
    \frac{\mathrm{d}}{\mathrm{d}t}
        \int_\Omega\left(
            \frac{\varepsilon|\nabla\varphi^\varepsilon|^2}{2}+\frac{W(\varphi^\varepsilon)}{2\varepsilon}
        \right)\,\mathrm{d}x
    =\int_\Omega\varepsilon\left(
        -\triangle\varphi^\varepsilon+\frac{W'(\varphi^\varepsilon)}{2\varepsilon^2}
    \right)\varphi_t^\varepsilon\,\mathrm{d}x,
\end{equation}
and a direct calculation yields that
\begin{align}
    \frac{\mathrm{d}}{\mathrm{d}t}E_p(t)
    &=\frac{\mathrm{d}}{\mathrm{d}t}\left[
        \frac{\alpha}{2}\left(
            \int_\Omega G(\varphi^\varepsilon)\,\mathrm{d}x
            -\int_\Omega G(\varphi_0^\varepsilon)\,\mathrm{d}x
        \right)^2
    \right]\\
    &=\alpha\left(
        \int_\Omega G(\varphi^\varepsilon)\,\mathrm{d}x
        -\int_\Omega G(\varphi_0^\varepsilon)\,\mathrm{d}x
    \right)\int_\Omega G'(\varphi^\varepsilon)\varphi_t^\varepsilon\,\mathrm{d}x.
\end{align}
Therefore, we have 
\begin{align}
    \frac{\mathrm{d}}{\mathrm{d}t}E(t)
    &=\frac{\mathrm{d}}{\mathrm{d}t}E_s(t)+\frac{\mathrm{d}}{\mathrm{d}t}E_p(t)\\
    &=\int_\Omega\left[
        -\varepsilon\triangle\varphi^\varepsilon
        +\frac{W'(\varphi^\varepsilon)}{2\varepsilon}
        +\alpha\left(
            \int_\Omega G(\varphi^\varepsilon)\,\mathrm{d}x
            -\int_\Omega G(\varphi_0^\varepsilon)\,\mathrm{d}x
        \right)G'(\varphi^\varepsilon)
    \right]\varphi_t^\varepsilon\,\mathrm{d}x\\
    &=\int_\Omega\left(
        -\varepsilon\varphi_t^\varepsilon+G'(\varphi^\varepsilon)\gamma(u^\varepsilon)
    \right)\varphi_t^\varepsilon\,\mathrm{d}x\\
    &=-\int_\Omega\varepsilon\left(\varphi_t^\varepsilon\right)^2\,\mathrm{d}x
    +\int_\Omega G'(\varphi^\varepsilon)\gamma(u^\varepsilon)\varphi_t^\varepsilon\,\mathrm{d}x.
\end{align}
It holds by definitions of $G$ and $W$ that 
\begin{equation}
    G'(\varphi^\varepsilon)=\sqrt{2W(\varphi^\varepsilon)}.
    \label{eq:Gp-W-relation}
\end{equation}
Here, note that for a solution $(\varphi^\varepsilon,u^\varepsilon)$ to \eqref{eq:target_problem}, it holds that $0\le\varphi^\varepsilon\le1$ and $u^\varepsilon\ge0$ by the maximum principle (see Theorem~\ref{thm:MP}).
Thus, we have
\begin{equation}
    \frac{\mathrm{d}}{\mathrm{d}t}E(t)
    =-\int_\Omega\varepsilon(\varphi_t^\varepsilon)^2\,\mathrm{d}x+\int_\Omega\sqrt{2W(\varphi^\varepsilon)}\gamma(u^\varepsilon)\varphi_t^\varepsilon\,\mathrm{d}x.
\end{equation}
Applying the Young's inequality, it holds that
\begin{align}
    \sqrt{2W(\varphi^\varepsilon)}\gamma(u^\varepsilon)\varphi_t^\varepsilon
    &\le\frac{1}{2}\frac{2W(\varphi
    ^\varepsilon)}{\varepsilon}\gamma^2(u^\varepsilon)+\frac{1}{2}\varepsilon\left(\varphi_t^\varepsilon\right)^2
    \le M_\gamma^2\frac{W(\varphi^\varepsilon)}{\varepsilon}+\frac{1}{2}\varepsilon\left(\varphi_t^\varepsilon\right)^2.
\end{align}
Therefore, it follows from the definition~\eqref{eq:energy} of energy $E$ that
\begin{align}
    \frac{\mathrm{d}}{\mathrm{d}t}E(t)
    &\le-\frac{1}{2}\int_\Omega\varepsilon\left(\varphi_t^\varepsilon\right)^2\,\mathrm{d}x+M_\gamma^2\int_\Omega\frac{W(\varphi^\varepsilon)}{\varepsilon}\,\mathrm{d}x\\
    &\le-\frac{1}{2}\int_\Omega\varepsilon\left(\varphi_t^\varepsilon\right)^2\,\mathrm{d}x
    +2M_\gamma^2E(t)\\
    &\le2M_\gamma^2E(t).
    \label{eq:dtdE}
\end{align}
Therefore, by using the Gronwall's inequality, we have the enegry estimate
\begin{equation}
    E(t)\le\mathrm{e}^{2M_\gamma^2t}E(0).
    \label{eq:E_bound}
\end{equation}
Since the inequality
\begin{equation}
    \frac{\mathrm{d}}{\mathrm{d}t}E(t)\le-\frac{1}{2}\int_\Omega\varepsilon(\varphi_t^\varepsilon)^2\,\mathrm{d}x+2M_\gamma^2E(t)
\end{equation}
holds by \eqref{eq:dtdE}, multiplying the both sides of the above inequality by $\mathrm{e}^{-2M_\gamma^2t}$, we have
\begin{equation}
    \frac{\mathrm{d}}{\mathrm{d}t}\left(E(t)\mathrm{e}^{-2M_\gamma^2t}\right)+\frac{1}{2}\mathrm{e}^{-2M_\gamma^2t}\int_\Omega\varepsilon\left(\varphi_t^\varepsilon\right)^2\,\mathrm{d}x\le0.
\end{equation}
Integrating the above inequality with respect to $t$ from $t=0$ to $t=T$, we obtain
\begin{equation}
    E(T)\mathrm{e}^{-2M_\gamma^2T}+\frac{1}{2}\int_0^T\mathrm{e}^{-2M_\gamma^2t}\int_\Omega\varepsilon\left(\varphi_t^\varepsilon\right)^2\,\mathrm{d}x\,\mathrm{d}t\le E(0).
\end{equation}
Since $\mathrm{e}^{-2M_\gamma^2t}\ge\mathrm{e}^{-2M_\gamma^2T}$ hold for all $t\in[0,T]$, we have
\begin{equation}
    E(T)+\int_0^T\int_\Omega\varepsilon\left(\varphi_t^\varepsilon\right)^2\,\mathrm{d}x\,\mathrm{d}t\le 2\mathrm{e}^{2M_\gamma^2T}E(0).
\end{equation}

(II)
The inequality~\eqref{eq:E_bound} yields that
\begin{equation}
    E_p(t)
    \le\mathrm{e}^{2M_\gamma^2t}E(0);
\end{equation}
that is,
\begin{equation}
    \left(
        \int_\Omega G(\varphi^\varepsilon)\,\mathrm{d}x
        -\int_\Omega G(\varphi_0^\varepsilon)\,\mathrm{d}x
    \right)^2
    \le\frac{2}{\alpha}\mathrm{e}^{2M_\gamma^2t}E(0).
\end{equation}
Therefore, fixing time $t$ arbitrarily and letting $\alpha\to\infty$, we see that
\begin{equation}
    \lim_{\alpha\to\infty}\int_\Omega G(\varphi^\varepsilon)\,\mathrm{d}x
    =\int_\Omega G(\varphi_0^\varepsilon)\,\mathrm{d}x.
\end{equation}

Using the relation~\eqref{eq:Gp-W-relation}, an elementary relation $(a+b)^2\le2(a^2+b^2)$, \eqref{eq:St_estimate}, and \eqref{eq:E_bound}, we have
\begin{align}
    &\int_0^T\int_\Omega\frac{1}{\varepsilon}\left[
        G'(\varphi^\varepsilon)\left(
            -\gamma(u^\varepsilon)+\tilde{S}[\varphi^\varepsilon](t)
        \right)
    \right]^2\,\mathrm{d}x\,\mathrm{d}t\\
    &\le\int_0^T\int_\Omega\frac{4W(\varphi^\varepsilon)}{\varepsilon}\left(
        \gamma^2(u^\varepsilon)+\tilde{S}[\varphi^\varepsilon](t)^2
    \right)\,\mathrm{d}x\,\mathrm{d}t\\
    &\le
    \left(M_\gamma^2+\frac{1}{9}\alpha^2|\Omega|^2\right)\int_0^T\int_\Omega\frac{4W(\varphi^\varepsilon)}{\varepsilon}\,\mathrm{d}x\,\mathrm{d}t\\
    &\le\left(M_\gamma^2+\frac{1}{9}\alpha^2|\Omega|^2\right)\int_0^T8E(t)\,\mathrm{d}t\\
    &\le\frac{4}{M_\gamma^2}\left(M_\gamma^2+\frac{1}{9}|\Omega|^2\alpha^2\right)\left(\mathrm{e}^{2M_\gamma^2T}-1\right)\mathrm{e}^{2M_\gamma^2T}E(0),
    \label{eq:l2estimate}
\end{align}
which implies that
\begin{equation}
    \sup_{\varepsilon>0}\int_0^T\int_\Omega\frac{1}{\varepsilon}\left[G'(\varphi^\varepsilon)\left(-\gamma(u^\varepsilon)+\tilde{S}[\varphi^\varepsilon](t)\right)\right]^2\,\mathrm{d}x\,\mathrm{d}t<\infty.
    \label{eq:rem-part_bounded}
\end{equation}
\end{proof}

\begin{remark}
    Although estimates in Lemma~\ref{lem:estimate} depend on $\alpha$, we will show that they are actually independent of $\alpha$ later.
\end{remark}

Lemma~\ref{lem:estimate} implies the following theorem.
See \cite{mugnai2008allen,mugnai2011convergence}  for details.

\begin{theorem}
    \label{thm:GM}
    Let $n=2,3$. 
    For each $\varepsilon>0$ and $t\in(0,T)$, define a Radon measure $\mu_t^\varepsilon$ on $\Omega$ by
    \begin{equation}
        \mathrm{d}\mu_t^\varepsilon
        =\left(
            \frac{\varepsilon|\nabla\varphi^\varepsilon|^2}{2}+\frac{W(\varphi^\varepsilon)}{2\varepsilon}
        \right)\,\mathrm{d}x.
    \end{equation}
    Then, there
    exists a family $\{\mu_t\}_{t\in[0,T)}$ of Radon measures on $\Omega$ such that the following hold.
\begin{enumerate}
    \item 
    There exists a sequence $\{\varepsilon_i\}_{i=1}^\infty$ such that $\varepsilon_i\searrow0$ and $\mu_t^{\varepsilon_i}\rightharpoonup\mu_t$ as $i\to\infty$; that is,
    \begin{equation}
        \int_\Omega\phi\,\mathrm{d}\mu_t^{\varepsilon_i}
        \longrightarrow
        \int_\Omega\phi\,\mathrm{d}\mu_t
    \end{equation}
    holds for any $\phi\in C_c(\Omega)$ and for any $t\geq 0$.
    Moreover,
    \begin{equation}
        \lim_{i\to\infty}\int_0^T\int_\Omega\left|\frac{\varepsilon_i|\nabla\varphi^{\varepsilon_i}|^2}{2}-\frac{W(\varphi^{\varepsilon_i})}{2\varepsilon_i}\right|\,\mathrm{d}x\,\mathrm{d}t=0
    \end{equation}
    holds.
    \item There exist a countably $(n-1)$-rectifiable set $M_t$ and $L_{\mathrm{loc}}^1$ function $\theta_t\colon M_t\rightarrow\mathbb{N}$ for a.e. $t\in(0,T)$ such that
    \begin{equation}
        \mu_t=c_0\theta_t\mathcal{H}^{n-1}|_{M_t};
    \end{equation}
    that is,
    \begin{equation}
        \int_\Omega\phi\,\mathrm{d}\mu_t
        =\int_{M_t}\phi c_0\theta_t\,\mathrm{d}\mathcal{H}^{n-1}
    \end{equation}
    holds for any $\phi\in C_c(\Omega)$.
    Here, the constant $c_0$ is given by
    \begin{equation}
        c_0\coloneqq\int_0^1\sqrt{W(s)}\,\mathrm{d}s=\frac{1}{6\sqrt{2}}.
    \end{equation}
    \item $\mu_t$ is an $L^2$-flow; that is,
    \begin{enumerate}
        \item there exists a square integrable generalized mean curvature vector $\vec{h}$, that is,
        \begin{align*}
            &\int_{0}^T\int_\Omega|\vec{h}|^2\,\mathrm{d}\mu_t\,\mathrm{d}t<\infty,\\
            &\int_\Omega\Div_{T_x\mu_t}\vec{\eta}\,\mathrm{d}\mu_t=-\int_\Omega
            \vec{\eta}\cdot\vec{h}\,\mathrm{d}\mu_t\quad\text{for all } \vec{\eta}\in C_c^1(\Omega;\mathbb{R}^n)\ \text{and a.e. } t\ge0,
        \end{align*}
        where $\Div_{T_x\mu_t}$ is the tangential divergence with respect to the approximate tangent space of $\mu_t$ (see \cite{mugnai2008allen}).
        \item there exists a generalized velocity vector $\vec{v}\in L_{\mathrm{loc}}^2((0,T);(L^2(\mu_t))^n)$; that is,
        \begin{enumerate}
            \item $\vec{v}$ is perpendicular to the approximate tangent space $T_x\mu_t$ for $\mu$-a.e. $(x,t)$, where $\mathrm{d}\mu=\mathrm{d}\mu_t\,\mathrm{d}$t;
            \item there exists a constant $C_T>0$ such that
            \begin{equation}
                \label{eq:estimate_eta}
                \left|
                    \int_0^T\int_\Omega\left(
                        \frac{\partial\eta}{\partial t}+\nabla\eta\cdot\vec{v}
                    \right)\,\mathrm{d}\mu_t\,\mathrm{d}t
                \right|
                \le C_T\|\eta\|_{C^0}
            \end{equation}
            holds for any $\eta\in C_c^1(\Omega\times(0,T))$;
            \item there exists $\vec{f}$ such that $\vec{v}=\vec{h}+\vec{f}$ holds.
        \end{enumerate}
    \end{enumerate}
\end{enumerate}
\end{theorem}

We derive an explicit form of $\vec{f}$ in (III-b-iii) later.

\begin{remark}
    Suppose that there exists a family of smooth hypersurfaces $\{M_t\}_{t\in[0,T)}$ and its velocity vector $\vec{v}$.
    Then, it holds that
    \begin{equation}
        \frac{\mathrm{d}}{\mathrm{d}t}\int_{M_t}\eta\,\mathrm{d}\mathcal{H}^{n-1}
        =\int_{M_t}\left[\left(-\eta\vec{h}+\nabla^\perp\eta\right)\cdot\vec{v}+\eta_t\right]\,\mathrm{d}\mathcal{H}^{n-1}
    \end{equation}
    for any $\eta\in C_0^1(\Omega\times(0,T))$ (see \cite{ecker2004regularity,tonegawa2019brakke}).
    Therefore, integrating the above relation in time yields
    \begin{equation}
        \left.\int_{M_t}\eta\,\mathrm{d}\mathcal{H}^{n-1}\right|_{t=0}^{t=T}
        =\int_0^T\int_{M_t}\left[\left(
            -\eta\vec{h}+\nabla^\perp\eta
        \right)\cdot\vec{v}+\eta_t\right]\,\mathrm{d}\mathcal{H}^{n-1}\,\mathrm{d}t.
    \end{equation}
    Since the left-hand side is equal to $0$ because of $\eta\in C_0^1(\Omega\times(0,T))$, we obtain
    \begin{equation}
        \left|
            \int_0^T\int_{M_t}\left(
                \nabla^\perp\eta\cdot\vec{v}+\eta_t
            \right)\,\mathrm{d}\mathcal{H}^{n-1}\,\mathrm{d}t
        \right|
        \le\|\eta\|_{C^0(\Omega\times(0,T))}\times\int_0^T\int_{M_t}|\vec{h}\cdot\vec{v}|\,\mathrm{d}\mathcal{H}^{n-1}\,\mathrm{d}t.
    \end{equation}
    Summarizing the above, we have shown that if $\{M_t\}$ is smooth and $\vec{v}$ is a normal velocity vector, then the estaimte
    \begin{equation}
        \left|
            \int_0^T\int_{M_t}\left(
                \eta_t+\nabla\eta\cdot\vec{v}
            \right)\,\mathrm{d}\mathcal{H}^{n-1}\,\mathrm{d}t
        \right|
        \le C_T\|\eta\|_{C^0(\Omega\times(0,T))}
        \label{eq:v_characterize}
    \end{equation}
    holds.
    Conversely, if \eqref{eq:v_characterize} holds for any $\eta$, there exists a normal velocity vector $\vec{v}$ of $M_t$~\cite{stuvard2024existence}.
\end{remark}

\begin{remark}
    From the first equation in \eqref{eq:target_problem}, we obtain
    \begin{equation}
        \frac{-\varphi_t^\varepsilon}{|\nabla\varphi^\varepsilon|}
        =\frac{-\triangle\varphi^\varepsilon+\frac{W'(\varphi^\varepsilon)}{2\varepsilon^2}}{|\nabla\varphi^\varepsilon|}-\frac{\sqrt{2W(\varphi^\varepsilon)}\left(\gamma(u^\varepsilon)-\tilde{S}[\varphi^\varepsilon]\right)}{\varepsilon|\nabla\varphi^\varepsilon|}
    \end{equation}
    on $\{|\nabla\varphi^\varepsilon|\neq0\}$.
    The left-hand side, the first and the second terms in the right-hand side correspond to normal velocity, mean curvature, and forcing terms, respectively.
    Since we can regard
    \begin{equation}
        \frac{\varepsilon|\nabla\varphi^\varepsilon|^2}{2}
        \approx\frac{W(\varphi^\varepsilon)}{2\varepsilon}
    \end{equation}
    in equilibrium state, the second term in the right-hand side is approximately equal to
    \begin{equation}
        -\sqrt{2}\left(\gamma(u^\varepsilon)-\tilde{S}[\varphi^\varepsilon]\right).
    \end{equation}
\end{remark}

\begin{lemma}
    \label{lem:phi_converge}
    There exists a subsequence of $\{\varphi^\varepsilon\}_{\varepsilon>0}$, denoted by the same symbol, and $\varphi\in BV(\Omega\times[0,T))\cap C^{1/2}([0,T);L^1(\Omega))$ such that 
    \begin{alignat}{2}
        &\varphi^\varepsilon\to\varphi&\qquad&\text{for a.e. } (x,t)\in\Omega\times(0,T),\\
        &\varphi^\varepsilon\to\varphi&\qquad&\text{in }\ L^1(\Omega\times(0,T)),\\
        &6G(\varphi^\varepsilon(\cdot,t))\to\varphi(\cdot,t)&\qquad&\text{in }\ L^1(\Omega)\ \text{ for all }\ t\ge0.
    \end{alignat}
    In particular, $\varphi(x,t)=0$ or $1$ for a.e. $(x,t)\in\Omega\times(0,\infty)$, and $\varphi (\cdot,t)=6G(\varphi(\cdot,t))=\chi_{\Omega_{\mathrm{in}}}(\varphi,t) \in BV (\Omega)$ for any $t\geq 0$.
\end{lemma}

The above lemma can be shown by the method of proving \cite[Proposition 8.3]{takasao2016existence}.
Here, note that estimates in Lemma~\ref{lem:estimate} play essential roles.

\begin{remark}
    \label{rem:nu}
    Introduce a measure $\nu_t$ by
    \begin{equation}
        \nu_t\coloneqq\theta_t\mathcal{H}^{n-1}|_{M_t},
        \label{eq:nu}
    \end{equation}
    where $M_t$ and $\theta_t$ are introduced in Theorem~\ref{thm:GM} (II).
    We see that
    \begin{equation}
        \nu_t
        =\frac{1}{c_0}\mu_t
        =6\sqrt{2}\mu_t.
        \label{eq:mu_nu}
    \end{equation}
    Note that $\nu_t$ is the surface measure for $\partial^\ast\Omega_{\mathrm{in}}(\varphi,t)$ if $\theta_t\equiv1$ \cite[Proposition 8.3]{takasao2016existence}, where $\partial^\ast\Omega_{\mathrm{in}}(\varphi,t)$ denotes the reduced boundary of $\Omega_{\mathrm{in}}(\varphi,t)$.
    Since $\mathcal{H}^{n-1}|_{\partial^\ast\Omega_{\mathrm{in}}(\varphi,t)}\ll\mu_t$ holds~\cite[Proposition 8.3]{takasao2016existence}, the Radon--Nikod\'ym derivative satisfies the relation
    \begin{equation}
        \frac{\mathrm{d}\mathcal{H}^{n-1}|_{\partial^\ast\Omega_{\mathrm{in}}(\varphi,t)}}{\mathrm{d}\mu_t}
        =6\sqrt{2}\frac{\mathrm{d}\mathcal{H}^{n-1}|_{\partial^\ast\Omega_{\mathrm{in}}(\varphi,t)}}{\mathrm{d}\nu_t}.
        \label{eq:RD-der_rel}
    \end{equation}
\end{remark}

\begin{proposition}
    Let $(\{u^\varepsilon\}_{\varepsilon>0},u)$ be a pair in Lemma~\ref{lem:u_converge} and $(\{\varphi^\varepsilon\}_{\varepsilon>0},\varphi)$ be a pair in Lemma~\ref{lem:phi_converge}.
    Suppose that there exists a subsequence of $\{u^\varepsilon\}_{\varepsilon>0}$, denoted by the same symbol, such that $\{u_0^\varepsilon\}_{\varepsilon>0}$ weakly converges to $u_0$ in $L^2(\Omega)$.
    Then, the function $u$ is a weak solution to
    \begin{equation}
        \begin{dcases*}
            \frac{\partial u}{\partial t}=\triangle u-ku+\varphi&in $\Omega\times(0,T)$,\\
            u|_{t=0}=u_0&in $\Omega$.
        \end{dcases*}
        \label{eq:u_limit}
    \end{equation}
\end{proposition}

\begin{proof}
    Let $\eta\in C_0^\infty(\Omega\times(0,T))$.
    Multiplying the second equation in \eqref{eq:target_problem} by $\eta$ and performing integration by parts, we obtain
    \begin{equation}
        \int_0^T\int_\Omega u_t^\varepsilon\eta\,\mathrm{d}x\,\mathrm{d}t
        =\int_0^T\int_\Omega\left(-\nabla u^\varepsilon\cdot\nabla\eta-ku^\varepsilon\eta+\varphi^\varepsilon\eta\right)\,\mathrm{d}x\,\mathrm{d}t.
    \end{equation}
    Since $u^\varepsilon$ weakly converges to $u$ in $H^1(\Omega\times(0,T))$ by Lemma~\ref{lem:u_converge}, we have
    \begin{align}
        &\int_0^T\int_\Omega u_t^\varepsilon\eta\,\mathrm{d}x\,\mathrm{d}t\longrightarrow\int_0^T\int_\Omega u_t\eta\,\mathrm{d}x\,\mathrm{d}t,\\
        &\int_0^T\int_\Omega\nabla u^\varepsilon\cdot\nabla\eta\,\mathrm{d}x\,\mathrm{d}t\longrightarrow\int_0^T\int_\Omega\nabla u\cdot\nabla\eta\,\mathrm{d}x\,\mathrm{d}t.
    \end{align}
    Since $u^\varepsilon$ strongly converges to $u$ in $L^2(\Omega\times(0,T))$ by Lemma~\ref{lem:u_converge}, it holds that
    \begin{equation}
        \int_0^T\int_\Omega u^\varepsilon\eta\,\mathrm{d}x\,\mathrm{d}t\longrightarrow\int_0^T\int_\Omega u\eta\,\mathrm{d}x\,\mathrm{d}t.
    \end{equation}
    Since $\varphi^\varepsilon$ strongly converges to $\varphi$ in $L^1(\Omega\times(0,T))$ by Lemma~\ref{lem:phi_converge}, we have
    \begin{equation}
        \int_0^T\int_\Omega\varphi^\varepsilon\eta\,\mathrm{d}x\,\mathrm{d}t\longrightarrow\int_0^T\int_\Omega\varphi\eta\,\mathrm{d}x\,\mathrm{d}t.
    \end{equation}
    Combining the above relations, we arrive at
    \begin{equation}
        \int_0^T\int_\Omega u_t\eta\,\mathrm{d}x\,\mathrm{d}t
        =\int_0^T\int_\Omega\left(-\nabla u\cdot\nabla\eta-ku\eta+\varphi\eta\right)\,\mathrm{d}x\,\mathrm{d}t.
    \end{equation}

    Concerning the initial value, since $u|_{t=0}\in H^{1/2}(\Omega)$ holds by the trace theorem, $u|_{t=0}$ is well-defined and the assumption yields that $u|_{t=0}=u_0$ holds.
\end{proof}

\section{Main results}

This section derives the interface equation, the singular-limit of \eqref{eq:target_problem}.

First we show the $L^2$-estimate for the non-local term. We note that the boundedness is clear from \eqref{eq:l2estimate} and the following estimate is not necessary for the existence theorem. However, 
we emphasise that the $L^2$-estimate does not depend on $\alpha$.

\begin{theorem}
Assume that there exists $C_1>0$ such that
\[
1 - \frac{6}{|\Omega|} \int _\Omega G(\varphi_0 ^\varepsilon) \,\mathrm{d}x \geq C_1,
\qquad
\int _\Omega G(\varphi_0 ^\varepsilon) \,\mathrm{d}x\geq C_1,
\qquad \varepsilon \in (0,1).
\]
Then, there exists a constant $C_2>0$ such that for sufficiently small $\varepsilon$ and large $\alpha$, it holds that
\[
\int _0 ^T |\tilde{S}[\varphi^\varepsilon](t)|^2 \,\mathrm{d}t \leq C_2 \left(1 + \mathrm{e}^{2M_\gamma^2T}E(0)\right )
\left( 
(1+M_\gamma ^2)T + C(T)
\right),
\]
where $C(T)$ is given by \eqref{eq:CT}.
\end{theorem}
\begin{remark}
    Since
    \begin{equation}
        \frac{6}{|\Omega|} \int _\Omega G(\varphi_0 ^\varepsilon) \,\mathrm{d}x
        \approx \frac{|\Omega_{\mathrm{in}}(\varphi^\varepsilon,0)|}{|\Omega|},
    \end{equation}
    the assumption means that $0<|\Omega_{\mathrm{in}}(\varphi^\varepsilon,0)|<|\Omega|$ holds.
\end{remark}
\begin{proof}
Let $\vec{\zeta} \in C^1 (\Omega;\mathbb{R}^n)$ be a smooth periodic 
vector-valued function.
Multiplying the first equation in \eqref{eq:target_problem} by $\nabla\varphi^\varepsilon\cdot\vec{\zeta}$ and integrating in space, we have
\begin{align}
\int_{\Omega} \varepsilon \varphi_t^\varepsilon\nabla \varphi ^\varepsilon \cdot \vec{\zeta}\,\mathrm{d}x 
&=  \int_\Omega \left( \varepsilon \triangle\varphi^\varepsilon-\frac{W'(\varphi^\varepsilon)}{2\varepsilon}\right) \nabla \varphi ^\varepsilon \cdot \vec{\zeta} \,\mathrm{d}x \\
&\quad +\int _\Omega G'(\varphi^\varepsilon) \gamma(u^\varepsilon) \nabla \varphi ^\varepsilon \cdot \vec{\zeta} \,\mathrm{d}x
-\tilde{S}[\varphi^\varepsilon](t) \int _\Omega G'(\varphi^\varepsilon) 
\nabla \varphi ^\varepsilon \cdot \vec{\zeta}\,\mathrm{d}x. 
\label{eq:IBP}
\end{align}    
The Cauchy--Schwarz inequality yields that
\begin{equation}
\left|
\int_{\Omega} \varepsilon\varphi_t^\varepsilon
\nabla \varphi ^\varepsilon \cdot \vec{\zeta}\,\mathrm{d}x
\right|
\leq \sqrt{2}\|\vec{\zeta}\|_{L^\infty(\Omega)}
\left(\int_\Omega\varepsilon\left(\varphi_t^\varepsilon\right)^2\,\mathrm{d}x\right) ^{\frac12} E(t)^{\frac12}.
\label{eq:estimate_1}
\end{equation}
By performing integration by parts, we have
\begin{align}
    \int_\Omega\left(\varepsilon\triangle\varphi^\varepsilon-\frac{W'(\varphi^\varepsilon)}{2\varepsilon}\right)\nabla\varphi^\varepsilon\cdot\vec{\zeta}\,\mathrm{d}x
    &=\int_\Omega\left[-\varepsilon\nabla\varphi^\varepsilon\cdot\nabla(\nabla\varphi^\varepsilon\cdot\vec{\zeta})+\frac{W(\varphi^\varepsilon)}{2\varepsilon}\Div\vec{\zeta}\right]\,\mathrm{d}x\\
    &=\int_\Omega\left[-\varepsilon\sum_{i,j=1}^n\left(\frac{\partial\varphi^\varepsilon}{\partial x_j}\frac{\partial^2\varphi^\varepsilon}{\partial x_j\partial x_i}\zeta_i+\frac{\partial\varphi^\varepsilon}{\partial x_j}\frac{\partial\varphi^\varepsilon}{\partial x_i}\frac{\partial\zeta_i}{\partial x_j}\right)+\frac{W(\varphi^\varepsilon)}{2\varepsilon}\Div\vec{\zeta}\right]\,\mathrm{d}x.
\end{align}
Performing integration by parts again yields that
\begin{equation}
    -\int_\Omega\frac{\partial\varphi^\varepsilon}{\partial x_j}\frac{\partial^2\varphi^\varepsilon}{\partial x_j\partial x_i}\zeta_i\,\mathrm{d}x
    =\int_\Omega\frac{\partial\varphi^\varepsilon}{\partial x_j}\left(\frac{\partial^2\varphi^\varepsilon}{\partial x_i\partial x_j}\zeta_i+\frac{\partial\varphi^\varepsilon}{\partial x_j}\frac{\partial\zeta_i}{\partial x_i}\right)\,\mathrm{d}x;
\end{equation}
that is,
\begin{equation}
    -\int_\Omega\sum_{i,j=1}^n\frac{\partial\varphi^\varepsilon}{\partial x_j}\frac{\partial^2\varphi^\varepsilon}{\partial x_j\partial x_i}\zeta_i\,\mathrm{d}x
    =\frac{1}{2}\int_\Omega\sum_{j=1}^n\left(\frac{\partial\varphi^\varepsilon}{\partial x_j}\right)^2\sum_{i=1}^n\frac{\partial\zeta_i}{\partial x_i}\,\mathrm{d}x
    =\frac{1}{2}\int_\Omega|\nabla\varphi^\varepsilon|^2\Div\vec{\zeta}\,\mathrm{d}x.
\end{equation}
Therefore, we obtain
\begin{align}
\int_\Omega \left( \varepsilon \triangle\varphi^\varepsilon-\frac{W'(\varphi^\varepsilon)}{2\varepsilon}\right) \nabla \varphi ^\varepsilon \cdot \vec{\zeta} \,\mathrm{d}x
&=\int_\Omega \left( \frac{\varepsilon|\nabla\varphi ^\varepsilon|^2}{2}+\frac{W(\varphi ^\varepsilon)}{2\varepsilon} \right) \mathrm{div} \, \vec{\zeta}
\,\mathrm{d}x - \int_\Omega \varepsilon \sum_{i,j=1} ^{n} \frac{\partial \varphi ^\varepsilon}{\partial x_i} \frac{\partial \varphi ^\varepsilon}{\partial x_j} \frac{\partial \zeta_i}{\partial x_j} \,\mathrm{d}x.
\end{align}
Then, the Cauchy--Schwarz inequality leads to
\begin{equation}
\left|\int_\Omega \left( \varepsilon \triangle\varphi^\varepsilon-\frac{W'(\varphi^\varepsilon)}{2\varepsilon}\right) \nabla \varphi ^\varepsilon \cdot \vec{\zeta} \,\mathrm{d}x
\right|
\leq C \| \nabla \vec{\zeta} \|_{L^\infty(\Omega)} E (t),
\label{eq:estimate_2}
\end{equation}
where $C>0$ depends only on $n$.
Furthermore, we have
\begin{equation}
\left|\int _\Omega G'(\varphi^\varepsilon) \gamma(u^\varepsilon) \nabla \varphi ^\varepsilon \cdot \vec{\zeta} \,\mathrm{d}x
\right|
\leq 2 \|\vec{\zeta}\|_{L^\infty(\Omega)}
M_\gamma E(t),
\label{eq:estimate_3}
\end{equation}
where we used $G'(\varphi ^\varepsilon ) |\nabla \varphi ^\varepsilon|
\leq \frac{\varepsilon|\nabla\varphi ^\varepsilon|^2}{2}+\frac{W(\varphi ^\varepsilon)}{\varepsilon}$ which follows from the Young's inequality.

Next, we show that for each $t\in(0,T)$, there exists $\vec{\zeta}\in C^1(\Omega;\mathbb{R}^n)$ such that
\begin{equation}
-\int _\Omega G'(\varphi^\varepsilon) 
\nabla \varphi ^\varepsilon \cdot \vec{\zeta}\,\mathrm{d}x
=\int _\Omega G(\varphi^\varepsilon) 
\Div\vec{\zeta}\,\mathrm{d}x
\geq\frac{C_1^2}{12},
\qquad t \in [0,T).
\label{eq:estimate_G-div_below}
\end{equation}
Let $\delta>0$ and $\rho_\delta$ be the standard mollifier.
Namely, let $\rho\in C_0^\infty(B_1(0))$ be a smooth non-negative function with $\int_{B_1(0)}\rho(x)\,\mathrm{d}x=1$, and define $\rho_\delta$ by $\rho_\delta(x)\coloneqq\delta^{-n}\rho(x/\delta)$.
Let
$w \in C^{2,\beta} (\Omega)$ be a periodic classical solution to
\begin{equation}
\begin{dcases*}
\Delta w = \rho_\delta\ast G(\varphi ^\varepsilon) 
- \frac{1}{|\Omega|} \int _\Omega \rho_\delta\ast G(\varphi ^\varepsilon) \,\mathrm{d}x&in $\Omega$,\\
\int_\Omega w \, \mathrm{d}x =0.
\end{dcases*}
\label{eq:w}
\end{equation}
Similar to the case of Neumann boundary condition, 
the solution exists and is unique. Set $\vec{\zeta}=\nabla w$.
Then we have 
\begin{align}
    \|\vec{\zeta}\|_{L^\infty(\Omega)},\|\nabla \vec{\zeta} \|_{L^\infty(\Omega)} \leq \|w\|_{C^{2,\beta}(\Omega)} \leq C (\delta),
\end{align}
where $C(\delta)$ depends only on $\delta$.
It follows from \eqref{eq:w} that
\begin{align}
    \int _\Omega G(\varphi^\varepsilon) 
\mathrm{div}\, \vec{\zeta}\,\mathrm{d}x
&=\int _\Omega G(\varphi^\varepsilon) 
\left(\rho_\delta\ast G(\varphi ^\varepsilon) 
- \frac{1}{|\Omega|} \int _\Omega \rho_\delta\ast G(\varphi ^\varepsilon) \,\mathrm{d}y\right)
\,\mathrm{d}x\\
&=
\int _\Omega G(\varphi^\varepsilon) ^2 \, \mathrm{d}x
+
\int _\Omega (\rho_\delta\ast G(\varphi^\varepsilon) -G(\varphi^\varepsilon))G(\varphi^\varepsilon)  \, \mathrm{d}x\\
&\quad- \frac{1}{|\Omega|} \int _\Omega \rho_\delta\ast G(\varphi ^\varepsilon) \, \mathrm{d}x
\int _\Omega G(\varphi^\varepsilon)  \, \mathrm{d}x.
\end{align}
We estimate each term below.
To estimate the first term, set
\begin{equation}
    \tilde{G}(\varphi)\coloneqq6G(\varphi)=3\varphi^2-2\varphi^3.
\end{equation}
A direct calculation yields that
\begin{equation}
    \tilde{G}(\varphi)^2-\tilde{G}(\varphi)=(-6-8\varphi+8\varphi^2)W(\varphi).
\end{equation}
Then, we have
\begin{align}
    \left|\int_\Omega\tilde{G}(\varphi^\varepsilon)^2-\int_\Omega\tilde{G}(\varphi^{\varepsilon})\,\mathrm{d}x\right|
    &=\left|\int_\Omega\left(-6+8\varphi^\varepsilon+8(\varphi^\varepsilon)^2\right)W(\varphi^\varepsilon)\,\mathrm{d}x\right|\\
    &\le22\int_\Omega\frac{W(\varphi^\varepsilon)}{2\varepsilon}\,\mathrm{d}x\cdot2\varepsilon
    \le44E(t)\varepsilon.
\end{align}
Therefore, we obtain
\begin{align}
    \int_\Omega G(\varphi^\varepsilon)^2\,\mathrm{d}x
    &=\frac{1}{36}\int_\Omega\tilde{G}(\varphi^\varepsilon)^2\,\mathrm{d}x
    \ge\frac{1}{36}\left(\int_\Omega\tilde{G}(\varphi^\varepsilon)\,\mathrm{d}x-44E(t)\varepsilon\right)\\
    &=\frac{1}{6}\int_\Omega G(\varphi^\varepsilon)\,\mathrm{d}x-\frac{11}{9}E(t)\varepsilon.
\end{align}
Note that
\begin{equation}
    \left|\int_K\left(\rho_\delta\ast v-v\right)\,\mathrm{d}x\right|\le\delta\|\nabla v\|(U)
\end{equation}
holds for all $\delta\in(0,\dist(K,\partial U))$, where $U\subset\mathbb{R}^d$ is an open set, $K\subset U$ is a compact set, and $u\in BV(K)$ \cite[Lemma~3.24]{ambrosio2000functions}.
In particular, for $v\in C^1(U)$, it holds that
\begin{equation}
    \int_K\left|\rho_\delta\ast v-v\right|\,\mathrm{d}x\le\delta\int_U|\nabla v|\,\mathrm{d}x.
\end{equation}
Therefore, combining with the maximum principle (Theorem~\ref{thm:MP}), we obtain
\begin{equation}
    \int_\Omega|\rho_\delta\ast G(\varphi^\varepsilon)-G(\varphi^\varepsilon)|\,\mathrm{d}x
    \le\delta\int_\Omega|\nabla(G(\varphi^\varepsilon))|\,\mathrm{d}x
    \le2\delta E(t)
    \label{eq:mollifier_estimate}
\end{equation}
for sufficiently small $\delta$, where we used $G'(\varphi ^\varepsilon ) |\nabla \varphi ^\varepsilon|
\leq \frac{\varepsilon|\nabla\varphi ^\varepsilon|^2}{2}+\frac{W(\varphi ^\varepsilon)}{\varepsilon}$ which follows from the Young's inequality.
Thus, we have
\begin{equation}
    \left|\int_\Omega\left(\rho_\delta\ast G(\varphi^\varepsilon)-G(\varphi^\varepsilon)\right)G(\varphi^\varepsilon)\,\mathrm{d}x\right|
    \le\frac{1}{3}\delta E(t).
\end{equation}
Similarly, using \eqref{eq:mollifier_estimate}, the third term can be bounded from below as
\begin{align}
&-\frac{1}{|\Omega|} \int _\Omega  \rho_\delta\ast G(\varphi ^\varepsilon) \,\mathrm{d}x
\int _\Omega G(\varphi^\varepsilon)  \, \mathrm{d}x\\
&=
-\frac{1}{|\Omega|} \left(\int _\Omega  G(\varphi ^\varepsilon) \,\mathrm{d}x\right)^2
-\frac{1}{|\Omega|} \int _\Omega (\rho_\delta\ast G(\varphi ^\varepsilon)-
G(\varphi ^\varepsilon))\,\mathrm{d}x
\int _\Omega G(\varphi^\varepsilon)  \,\mathrm{d}x \\
&\geq  -\frac{1}{|\Omega|} \left(\int _\Omega  G(\varphi ^\varepsilon) \,\mathrm{d}x\right)^2
- \frac{1}{3}\delta E(t).
\end{align}
Hence, using the assumption and \eqref{eq:E_bound}, we have
\begin{align}
    &\int_\Omega G(\varphi^\varepsilon)\Div\vec{\zeta}\,\mathrm{d}x\\
    &\ge\frac{1}{6}\int_\Omega G(\varphi^\varepsilon)\,\mathrm{d}x-\frac{11}{9}E(t)\varepsilon-\frac{1}{3}\delta E(t)-\frac{1}{|\Omega|}\left(\int_\Omega G(\varphi^\varepsilon)\,\mathrm{d}x\right)^2-\frac{1}{3}\delta E(t)\\
    &=\frac{1}{6}\int_\Omega G(\varphi^\varepsilon)\,\mathrm{d}x\left(1-\frac{6}{|\Omega|}\int_\Omega G(\varphi^\varepsilon)\,\mathrm{d}x\right)-\frac{1}{3}\left(\frac{11}{3}\varepsilon+2\delta\right)E(t)\\
    &=\frac{1}{6}\int_\Omega G(\varphi^\varepsilon)\,\mathrm{d}x\left(1-\frac{6}{|\Omega|}\int_\Omega G(\varphi_0^\varepsilon)\,\mathrm{d}x+\frac{6}{|\Omega|}\int_\Omega\left(G(\varphi_0^\varepsilon)-G(\varphi^\varepsilon)\right)\,\mathrm{d}x\right)\\
    &\quad-\frac{1}{3}\left(\frac{11}{3}\varepsilon+2\delta\right)E(t)
   \\
    &\ge\frac{1}{6}\int_\Omega G(\varphi^\varepsilon)\,\mathrm{d}x\left(C_1-\frac{6}{|\Omega|}\sqrt{\frac{2}{\alpha}}\mathrm{e}^{M_\gamma^2T}E(0)^{1/2}\right)-\frac{1}{3}\left(\frac{11}{3}\varepsilon+2\delta\right)\mathrm{e}^{2M_\gamma^2T}E(0)\\
    &=\frac{1}{6}\left(C_1-\sqrt{\frac{2}{\alpha}}\mathrm{e}^{M_\gamma^2T}E(0)^{1/2}\right)\left(C_1-\frac{6}{|\Omega|}\sqrt{\frac{2}{\alpha}}\mathrm{e}^{M_\gamma^2T}E(0)^{1/2}\right)\\
    &\quad-\frac{1}{3}\left(\frac{11}{3}\varepsilon+2\delta\right)\mathrm{e}^{2M_\gamma^2T}E(0)\\
    &=\frac{1}{6}\left(C_1^2+\frac{12}{|\Omega|\alpha}\mathrm{e}^{2M_\gamma^2T}E(0)\right)\\
    &\quad-\left[C_1\left(\frac{6}{|\Omega|}+1\right)\sqrt{\frac{2}{\alpha}}+\frac{1}{3}\left(\frac{11}{3}\varepsilon+2\delta\right)\mathrm{e}^{M_\gamma^2T}E(0)^{1/2}\right]\mathrm{e}^{M_\gamma^2T}E(0)^{1/2},
\end{align}
where we have used the assumption and the energy estimate~\eqref{eq:E_bound}.
Therefore, for sufficiently small $\delta$, $\varepsilon$ and large $\alpha$, we have
\[
\int _\Omega G(\varphi^\varepsilon) 
\mathrm{div}\, \vec{\zeta}\, \mathrm{d}x
\geq\frac{C_1^2}{12}>0,
\]
which completes the proof of \eqref{eq:estimate_G-div_below}.

Thus, combining \eqref{eq:IBP} with estimates \eqref{eq:estimate_1}, \eqref{eq:estimate_2}, \eqref{eq:estimate_3}, and \eqref{eq:estimate_G-div_below} and using the Young's inequality, we obtain
\begin{align*}
&|\tilde{S}[\varphi^\varepsilon](t)|\\
&\leq
\frac{12}{C_1^2}
\left(\sqrt{2}
\|\vec{\zeta}\|_{L^\infty(\Omega)}
\left(\int_\Omega\varepsilon\left(\varphi_t^\varepsilon\right)^2\,\mathrm{d}x\right) ^{\frac12} E(t)^{\frac12}
+
C \| \nabla \vec{\zeta} \|_{L^\infty(\Omega)} E (t)
+
2 \|\vec{\zeta}\|_{L^\infty(\Omega)}
M_\gamma E(t)
\right)\\
&\le\frac{12}{C_1^2}
\| \vec{\zeta} \|_{C^1 (\Omega)} \left(1 + \mathrm{e}^{2M_\gamma^2t}E(0)\right)
\left(\frac{1}{\sqrt{2}}
\left(\int_\Omega\varepsilon\left(\varphi_t^\varepsilon\right)^2\,\mathrm{d}x\right) ^{\frac12}
+ C + 2M_\gamma
\right).
\end{align*}
Hence, we conclude that
\begin{equation}
\begin{split}
&\int_0 ^T 
|\tilde{S}[\varphi^\varepsilon](t)|^2 \, \mathrm{d}t\\
&\leq\frac{6}{C_1^2}\| \vec{\zeta} \|_{C^1 (\Omega)} \left(1 + \mathrm{e}^{2M_\gamma^2T}E(0)\right)
\left[
(C+2M_\gamma)^2T +\frac{1}{2}\int _0 ^T \int_\Omega\varepsilon\left(\frac{\partial\varphi^\varepsilon}{\partial t}\right)^2\,\mathrm{d}x\,\mathrm{d}t
\right] \\
&\leq C_2\| \vec{\zeta} \|_{C^1 (\Omega)} \left(1 + \mathrm{e}^{2M_\gamma^2T}E(0)\right)
\left( 
(1+M_\gamma ^2)T + C(T)
\right),
\end{split}
\end{equation}
where the first and second inequalities follow from the Young's inequality and Lemma~\ref{lem:estimate}(I), respectively.
\end{proof}

Define $\vec{v}^\varepsilon$ by
\begin{equation}
    \vec{v}^\varepsilon
    \coloneqq\begin{dcases*}
        \frac{-\varphi_t^\varepsilon}{|\nabla\varphi^\varepsilon|}\cdot\frac{\nabla\varphi^\varepsilon}{|\nabla\varphi^\varepsilon|}&if $|\nabla\varphi^\varepsilon|\neq0$,\\
        \vec{0}&otherwise
    \end{dcases*}
\end{equation}
and set
\begin{equation}
    \mathrm{d}\tilde{\mu}_t^\varepsilon
    \coloneqq\varepsilon|\nabla\varphi^\varepsilon|^2\,\mathrm{d}x.
\end{equation}

\begin{lemma}
    \label{claim:convergence}
    Let $\{\varepsilon_i\}_{i=1}^\infty$ be a sequence as in Theorem~\ref{thm:GM}.
    Then, the following hold.
    \begin{enumerate}
        \item $\displaystyle\sup_{\varepsilon>0}\int_0^T\int_\Omega|\vec{v}^\varepsilon|^2\,\mathrm{d}\tilde{\mu}_t^\varepsilon\,\mathrm{d}t<\infty$.
        \item $\displaystyle\mathrm{d}\tilde{\mu}_t^{\varepsilon_i}\,\mathrm{d}t\rightharpoonup\,\mathrm{d}\mu_t\,\mathrm{d}t$ holds as $i\to\infty$; that is,
        \begin{equation}
            \int_0^T\int_\Omega\eta\,\mathrm{d}\tilde{\mu}_t^{\varepsilon_i}\,\mathrm{d}t
            \longrightarrow\int_{\Omega_T}\eta\,\mathrm{d}\mu_t\,\mathrm{d}t\qquad
            \text{for any}\ \eta\in C_0(\Omega\times(0,T)).
        \end{equation}
        \item There exist some subsequence of $\{\varepsilon_i\}_{i=1}^\infty$, denoted by same symbols, and $\vec{v}\in(L^2(\mu;\Omega\times(0,T)))^n$ such that
        \begin{equation}
            \int_0^T\int_\Omega\vec{v}^{\varepsilon_i}\cdot\vec{\eta}\,\mathrm{d}\tilde{\mu}_t^{\varepsilon_i}\,\mathrm{d}t
            \longrightarrow\int_0^T\int_\Omega\vec{v}\cdot\vec{\eta}\,\mathrm{d}\mu_t\,\mathrm{d}t\qquad\text{for any}\ \vec{\eta}\in C_0(\Omega\times(0,T);\mathbb{R}^n).
        \end{equation}
    \end{enumerate}
\end{lemma}

Note that the vector $\vec{v}$ in (III) of Lemma~\ref{claim:convergence} is nothing but the one in Theorem~\ref{thm:GM} (III-b).
The statement in (III) follows from the compactness of $\{(\mu^\varepsilon,\vec{v}^\varepsilon)\}_{\varepsilon}$.

\begin{proof}
    (I)
    Since the definition of $\vec{v}^\varepsilon$ and Lemma~\ref{lem:estimate} (I) imply that
    \begin{align}
        \int_0^T\int_\Omega|\vec{v}^\varepsilon|^2\,\mathrm{d}\tilde{\mu}_t^\varepsilon\,\mathrm{d}t
        &=\int_0^T\int_{\{x\in\Omega\mid|\nabla\varphi^\varepsilon(x,t)|\neq0\}}\frac{(\varphi_t^\varepsilon)^2}{|\nabla\varphi^\varepsilon|^2}\varepsilon|\nabla\varphi^\varepsilon|^2\,\mathrm{d}x\,\mathrm{d}t\\
        &\le\int_0^T\int_\Omega\varepsilon(\varphi_t^\varepsilon)^2\,\mathrm{d}x\,\mathrm{d}t
        \le C(T),
    \end{align}
    we obtain the desired estimate.

    (II)
    Since it holds that
    \begin{align}
        \int_0^T\int_\Omega\eta\,\mathrm{d}\tilde{\mu}_t^{\varepsilon_i}\,\mathrm{d}t
        &=\int_0^T\int_\Omega\eta\left(\frac{\varepsilon_i|\nabla\varphi^{\varepsilon_i}|^2}{2}+\frac{W(\varphi^{\varepsilon_i})}{2\varepsilon_i}+\frac{\varepsilon_i|\nabla\varphi^{\varepsilon_i}|^2}{2}-\frac{W(\varphi^{\varepsilon_i})}{2\varepsilon_i}\right)\,\mathrm{d}x\,\mathrm{d}t\\
        &=\int_0^T\int_\Omega\eta\,\mathrm{d}\mu_t^{\varepsilon_i}\,\mathrm{d}t+\int_0^T\int_\Omega\eta\left(\frac{\varepsilon_i|\nabla\varphi^{\varepsilon_i}|^2}{2}-\frac{W(\varphi^{\varepsilon_i})}{2\varepsilon_i}\right)\,\mathrm{d}x\,\mathrm{d}t.
    \end{align}
    Theorem~\ref{thm:GM} (I) yields the statement.

    Concerning (III), see \cite{hutchinson1986second}.
\end{proof}

\begin{theorem}
    Let $n=2,3$.
    There exists a sequence $\{\varepsilon_i\}_{i=1}^\infty$ such that the limits $\varphi\coloneqq\lim_{i\to\infty}\varphi^{\varepsilon_i}$ and $u\coloneqq\lim_{i\to\infty}u^{\varepsilon_i}$ exist, and the generalized velocity vector $\vec{v}$ is given by
    \begin{equation}
        \vec{v}
    =\vec{h}-\sqrt{2}\tilde{S}[\varphi]\left(\frac{\mathrm{d}\mathcal{H}^{n-1}|_{\partial^\ast\Omega_{\mathrm{in}}(\varphi,t)}}{\mathrm{d}\nu_t}\right)\vec{\nu}+\sqrt{2}\gamma(u)\left(
        \frac{\mathrm{d}\mathcal{H}^{n-1}|_{\partial^\ast\Omega_{\mathrm{in}}(\varphi,t)}}{\mathrm{d}\nu_t}
    \right)\vec{\nu};
    \end{equation}
    that is,
\begin{align*}
    &\int_0^T\int_\Omega\vec{v}\cdot\vec{\eta}\,\mathrm{d}\nu_t\,\mathrm{d}t\\
    &=\int_0^T\int_\Omega\vec{h}\cdot\vec{\eta}\,\mathrm{d}\nu_t\,\mathrm{d}t\\
    &\quad+\int_0^T\int_\Omega\left[-\sqrt{2}\tilde{S}[\varphi]\left(\frac{\mathrm{d}\mathcal{H}^{n-1}|_{\partial^\ast\Omega_{\mathrm{in}}(\varphi,t)}}{\mathrm{d}\nu_t}\right)\vec{\nu}+\sqrt{2}\gamma(u)\left(
        \frac{\mathrm{d}\mathcal{H}^{n-1}|_{\partial^\ast\Omega_{\mathrm{in}}(\varphi,t)}}{\mathrm{d}\nu_t}
    \right)\vec{\nu}\right]\cdot\vec{\eta}\,\mathrm{d}\nu_t\,\mathrm{d}t
\end{align*}
holds for any $\vec{\eta}\in(C_0^1(\Omega\times(0,T)))^n$, where the measure $\nu_t$ is defined by \eqref{eq:nu} and $\vec{\nu}$ is the outer unit normal vector of $\partial ^\ast \Omega_{\mathrm{in}} (\varphi,t)$.
Furthermore, $u$ is a weak solution to \eqref{eq:u_limit}.
\end{theorem}

\begin{remark}
    If we rewrite the above theorem in terms of $\mu_t$ using \eqref{eq:mu_nu}, we obtain Theorem \ref{thm:main}.
\end{remark}

\begin{proof}
By the definition of $\vec{v}^\varepsilon$ and the first equation in \eqref{eq:target_problem}, it holds that
\begin{align}
    &\int_0^T\int_\Omega\vec{v}^{\varepsilon_i}\cdot\vec{\eta}\,\mathrm{d}\tilde{\mu}_t^{\varepsilon_i}\,\mathrm{d}t\\
    &=\int_0^T\int_{\{x\in\Omega\mid\nabla\varphi^{\varepsilon_i}(x,t)\neq0\}}\frac{-\varphi_t^{\varepsilon_i}}{|\nabla\varphi^{\varepsilon_i}|}\frac{\nabla\varphi^{\varepsilon_i}}{|\nabla\varphi^{\varepsilon_i}|}\cdot\vec{\eta}{\varepsilon_i}|\nabla\varphi^{\varepsilon_i}|^2\,\mathrm{d}x\,\mathrm{d}t\\
    &=-\int_0^T\int_\Omega{\varepsilon_i}\varphi_t^{\varepsilon_i}\nabla\varphi^{\varepsilon_i}\cdot\vec{\eta}\,\mathrm{d}x\,\mathrm{d}t\\
    &=-\int_0^T\int_\Omega\left(\varepsilon\triangle\varphi^\varepsilon-\frac{W'(\varphi^\varepsilon)}{2\varepsilon}-\sqrt{2W(\varphi^\varepsilon)}(-\gamma(u^\varepsilon)+\tilde{S}[\varphi^\varepsilon])\right)\nabla\varphi^\varepsilon\cdot\vec{\eta}\,\mathrm{d}x\,\mathrm{d}t\\
    &=\int_0^T\int_\Omega\left(-\varepsilon\triangle\varphi^\varepsilon+\frac{W'(\varphi^\varepsilon)}{2\varepsilon}\right)\nabla\varphi^\varepsilon\cdot\vec{\eta}\,\mathrm{d}x\,\mathrm{d}t\\
    &\quad+\int_0^T\int_\Omega\sqrt{2W(\varphi^\varepsilon)}\left(-\gamma(u^\varepsilon)+\tilde{S}[\varphi^\varepsilon]\right)\nabla\varphi^\varepsilon\cdot\vec{\eta}\,\mathrm{d}x\,\mathrm{d}t\\
    &\equiv I_1+I_2
\end{align}
for any $\vec{\eta}\in C_0^1(\Omega\times(0,T);\mathbb{R}^n)$.
Let
\begin{equation}
    V_t^{\varepsilon_i}(\vec{\phi})
    \coloneqq\int_{\{x\in\Omega\mid\nabla\varphi^{\varepsilon_i}(x,t)\neq0\}}\vec{\phi}\left(x,I-\frac{\nabla\varphi^{\varepsilon_i}}{|\nabla\varphi^{\varepsilon_i}|}\otimes\frac{\nabla\varphi^{\varepsilon_i}}{|\nabla\varphi^{\varepsilon_i}|}\right)\,\mathrm{d}\mu_t^{\varepsilon_i}
\end{equation}
for $\vec{\eta}\in C_0(\Omega\times G(n-1,n))$, where $G(n-1,n)$ denotes  the Grassmannian; that is, the subspace of all hyperplanes in $\mathbb{R}^n$.
Note that $V_t^{\varepsilon_i}$ is a natural varifold for our problem~\eqref{eq:target_problem}~\cite{takasao2016existence}.
In \cite[Lemma 6.6]{takasao2016existence}, it has been shown that
\begin{equation}
    I_1
    =-\int_0^T\delta V_t^{\varepsilon_i}(\vec{\eta})\,\mathrm{d}t+\text{(error term)}
    \longrightarrow
    \int_0^T\int_\Omega\vec{h}\cdot\vec{\eta}\,\mathrm{d}\mu_t\,\mathrm{d}t,
\end{equation}
which follows from Theorem~\ref{thm:GM} (I), where $\delta V_t^\varepsilon(\vec{\eta})$ denotes the first variation of $V_t^\varepsilon$.
Concerning $I_2$, note that 
\begin{equation}
    \sqrt{2W(\varphi^\varepsilon)}\nabla\varphi^\varepsilon
    =G'(\varphi^\varepsilon)\nabla\varphi^\varepsilon
    =\nabla(G(\varphi^\varepsilon))
\end{equation}
holds.
Then, we have
\begin{align}
    I_2
    &=-\int_0^T\int_\Omega\left(
        \gamma(u^{\varepsilon_i})-\tilde{S}[\varphi^{\varepsilon_i}]
    \right)\nabla\left(G(\varphi^{\varepsilon_i})\right)\cdot\vec{\eta}\,\mathrm{d}x\,\mathrm{d}t\\
    &=-\int_0^T\int_\Omega\gamma(u^{\varepsilon_i})\nabla \left(G(\varphi^{\varepsilon_i})\right)\cdot\vec{\eta}\,\mathrm{d}x\,\mathrm{d}t
    +\int_0^T\int_\Omega\tilde{S}[\varphi^{\varepsilon_i}]\nabla \left(G(\varphi^{\varepsilon_i})\right)\cdot\vec{\eta}\,\mathrm{d}x\,\mathrm{d}t\\
    &\equiv J_1+J_2.
\end{align}
By integration by parts, $J_2$ can be calculated as
\begin{equation}
    J_2
    =\int_0^T\tilde{S}[\varphi^{\varepsilon_i}]\int_\Omega\nabla G(\varphi^{\varepsilon_i})\cdot\vec{\eta}\,\mathrm{d}x\,\mathrm{d}t
    =-\int_0^T\tilde{S}[\varphi^{\varepsilon_i}]\int_\Omega G(\varphi^{\varepsilon_i})\Div\vec{\eta}\,\mathrm{d}x\,\mathrm{d}t.
\end{equation}
Since $6G(\varphi^\varepsilon(\cdot,t))\to\varphi(\cdot,t)$ in $L^1(\Omega)$ for all $t\ge0$ by Lemma \ref{lem:phi_converge}, it holds that
\begin{align}
    \tilde{S}[\varphi^\varepsilon](t)
    \to\tilde{S}[\varphi](t)
    =\alpha\left(\int_\Omega\frac{1}{6}\varphi(x,t)\,\mathrm{d}x-\int_\Omega\frac{1}{6}\varphi(x,0)\,\mathrm{d}x\right)
\end{align}
for all $t\ge0$ and $\sup_{i\in\mathbb{N}}|\tilde{S}[\varphi^{\varepsilon_i}]|<\infty$. Therefore, we obtain
\begin{align}
    \tilde{S}[\varphi^{\varepsilon_i}]\longrightarrow\tilde{S}[\varphi]\qquad
    \text{in }\ L^2(0,T).
\end{align}
Similarly, we have
\begin{equation}
    \int_\Omega G(\varphi^{\varepsilon_i})\Div\vec{\eta}\,\mathrm{d}x
    \longrightarrow
    \int_\Omega G(\varphi)\Div\vec{\eta}\,\mathrm{d}x
\end{equation}
for all $t\ge0$
and $\sup_{i\in\mathbb{N}}\|G(\varphi^{\varepsilon_i})\|_{L^\infty(\Omega\times(0,T))}<\infty$.
Therefore, we obtain
\begin{equation}
    \int_\Omega G(\varphi^{\varepsilon_i})\Div\vec{\eta}\,\mathrm{d}x
    \longrightarrow\int_\Omega G(\varphi)\Div\vec{\eta}\,\mathrm{d}x\qquad
    \text{in}\ L^2(0,T).
\end{equation}
Hence, we obtain
\begin{align}
    J_2
    &\longrightarrow-\int_0^T\tilde{S}[\varphi]\int_\Omega G(\varphi)\Div\vec{\eta}\,\mathrm{d}x\,\mathrm{d}t
    =-\frac{1}{6}\int_0^T\tilde{S}[\varphi]\int_{\Omega_{\mathrm{in}}(\varphi,t)}\Div\vec{\eta}\,\mathrm{d}x\,\mathrm{d}t\\
    &=-\frac{1}{6}\int_0^T\tilde{S}[\varphi]\int_{\partial^\ast\Omega_{\mathrm{in}}(\varphi,t)}\vec{\eta}\cdot\vec{\nu}\,\mathrm{d}\mathcal{H}^{n-1}\,\mathrm{d}t.
\end{align}
Therefore, using the Radon--Nikod\'ym derivative, we obtain that
\begin{equation}
    -\frac{1}{6}\int_0^T\tilde{S}[\varphi]\int_{\partial^\ast\Omega_{\mathrm{in}}(\varphi,t)}\vec{\eta}\cdot\vec{\nu}\,\mathrm{d}\mathcal{H}^{n-1}\,\mathrm{d}t
    =-\frac{1}{6}\int_0^T\tilde{S}[\varphi]\int_\Omega\vec{\eta}\cdot\vec{\nu}\left(
        \frac{\mathrm{d}\mathcal{H}^{n-1}|_{\partial^\ast\Omega_{\mathrm{in}}(\varphi,t)}}{\mathrm{d}\mu_t}
    \right)\,\mathrm{d}\mu_t\,\mathrm{d}t.
\end{equation}
Concerning $J_1$, performing integration by parts, we obtain
\begin{align}
    J_1
    &=-\int_0^T\int_\Omega\gamma(u^{\varepsilon_i})\nabla G(\varphi^{\varepsilon_i})\cdot\vec{\eta}\,\mathrm{d}x\,\mathrm{d}t
    =\int_0^T\int_\Omega G(\varphi^{\varepsilon_i})\Div(\gamma(u^{\varepsilon_i})\cdot\vec{\eta})\,\mathrm{d}x\,\mathrm{d}t\\
    &=\int_0^T\int_\Omega G(\varphi^{\varepsilon_i})\left[
        \gamma'(u^{\varepsilon_i})\nabla u^{\varepsilon_i}\cdot\vec{\eta}
        +\gamma(u^{\varepsilon_i})\Div\vec{\eta}
    \right]\,\mathrm{d}x\,\mathrm{d}t\\
    &\equiv
    K_1+K_2.
\end{align}
For $K_1$, note that by taking a subsequence of $\{\varepsilon_i\}_{i=1}^\infty$ if necessary, it holds that $\{u^{\varepsilon_i}\}$ weakly converges to $u$ in $H^1(\Omega\times(0,T))$ and strongly in $L^2(\Omega\times(0,T))$ by Lemma~\ref{lem:u_converge}.
Let us consider to estimate
\begin{align}
    &\left|\int_0^T\int_\Omega G(\varphi^{\varepsilon_i})\gamma
    '(u^{\varepsilon_i})\nabla u^{\varepsilon_i}\cdot\vec{\eta}\,\mathrm{d}x\,\mathrm{d}t-\int_0^T\int_\Omega
     G(\varphi)\gamma'(u)\nabla u\cdot\vec{\eta}\,\mathrm{d}x\,\mathrm{d}t\right|\\
     &\le\left|\int_0^T\int_\Omega(G(\varphi^{\varepsilon_i})-G(\varphi))\gamma'(u^{\varepsilon_i})\nabla u^{\varepsilon_i}\cdot\vec{\eta}\,\mathrm{d}x\,\mathrm{d}t\right|\\
     &\qquad+\left|\int_0^T\int_\Omega G(\varphi)(\gamma'(u^{\varepsilon_i})-\gamma'(u))\nabla u^{\varepsilon_i}\cdot\vec{\eta}\,\mathrm{d}x\,\mathrm{d}t\right|\\
     &\qquad+\left|\int_0^T\int_\Omega G(\varphi)\gamma'(u)\nabla(u^{\varepsilon_i}-u)\cdot\vec{\eta}\,\mathrm{d}x\,\mathrm{d}t\right|\\
     &\equiv L_1+L_2+L_3.
\end{align}
Concerning $L_1$, since $\gamma'(u^{\varepsilon_i})$ is bounded, by estimating $\vec{\eta}$ by its $L^\infty$-norm and applying the Cauchy--Schwarz inequality, we have
\begin{align}
    L_1
    &\le C\left(\int_0^T\int_\Omega(G(\varphi^{\varepsilon_i})-G(\varphi))^2\,\mathrm{d}x\,\mathrm{d}t\right)^{1/2}\left(\int_0^T\int_\Omega|\nabla u^{\varepsilon_i}|^2\,\mathrm{d}x\,\mathrm{d}t\right)^{1/2}.
\end{align}
Note that $\{\|\nabla u^{\varepsilon_i}\|_{L^2(\Omega\times(0,T))}\}_{i=1}^\infty$ is bounded, $|G(\varphi^{\varepsilon_i})-G(\varphi)|\le C$ since $\{\varphi^{\varepsilon_i}\}_{i=1}^\infty$ and $\varphi$ are bounded.
Moreover, $\varphi^{\varepsilon_i}\to\varphi$ for a.e.~$(x,t)\in\Omega\times(0,T)$ by Lemma \ref{lem:phi_converge}.
Thus, the Lebesgue's dominated convergence theorem implies that $L_1\to0$ as $i\to\infty$.
As for $L_2$, since $G(\varphi)\in L^\infty(\Omega\times(0,T))$, the Cauchy--Schwarz inequality yields that
\begin{align}
    L_2
    \le C\left(\int_0^T\int_\Omega(\gamma'(u^{\varepsilon_i})-\gamma'(u))^2\,\mathrm{d}x\,\mathrm{d}t\right)^{1/2}\left(\int_0^T\int_\Omega|\nabla u^{\varepsilon_i}|^2\,\mathrm{d}x\,\mathrm{d}t\right)^{1/2}.
\end{align}
Since $\{\|\nabla u^{\varepsilon_i}\|_{L^2(\Omega\times(0,T))}\}_{i=1}^\infty$ is bounded and $\{u^{\varepsilon_i}\}_{i=1}^\infty$ strongly converges to $u$ in $L^2(\Omega\times(0,T))$, we have $L_2\to0$ as $i\to\infty$.
Finally, since $G(\varphi)\gamma'(u)\vec{\eta}\in L^2(\Omega\times(0,T))$ and $\{u^{\varepsilon_i}\}_{i=1}^\infty$ weakly converges to $u$ in $H^1(\Omega\times(0,T))$, we see that $L_3\to0$.
Therefore, we have
\begin{align}
    K_1
    &=\int_0^T\int_\Omega G(\varphi^{\varepsilon_i})\gamma'(u^{\varepsilon_i})\nabla u^{\varepsilon_i}\cdot\vec{\eta}\,\mathrm{d}x\,\mathrm{d}t\\
    &\longrightarrow
    \int_0^T\int_{\Omega} G(\varphi)\gamma'(u)\nabla u\cdot\vec{\eta}\,\mathrm{d}x\,\mathrm{d}t
    =\frac{1}{6}\int_0^T\int_{\Omega_{\mathrm{in}}(\varphi,t)}\gamma'(u)\nabla u\cdot\vec{\eta}\,\mathrm{d}x\,\mathrm{d}t
\end{align}
holds, where we have used $G(u^{\varepsilon_i})\rightarrow G(u)$ and $\gamma'(u^{\varepsilon_i})\rightarrow\gamma'(u)$ as $i\to\infty$ in $L^2(\Omega\times(0,T))$.
For $K_2$, we obtain in a similar manner that
\begin{align}
    K_2
    &=\int_0^T\int_\Omega G(\varphi^{\varepsilon_i})\gamma(u^{\varepsilon_i})\Div\vec{\eta}\,\mathrm{d}x\,\mathrm{d}t\\
    &\longrightarrow
    \int_0^T\int_\Omega G(\varphi)\gamma(u)\Div\vec{\eta}\,\mathrm{d}x\,\mathrm{d}t
    =\frac{1}{6}\int_0^T\int_{\Omega_{\mathrm{in}}(\varphi,t)}\gamma(u)\Div\vec{\eta}\,\mathrm{d}x\,\mathrm{d}t
\end{align}
holds.
As a result, we obtain
\begin{align}
    K_1+K_2
    &\longrightarrow\frac{1}{6}\int_0^T\int_{\Omega_{\mathrm{in}}(\varphi,t)}\left[
        \gamma'(u)\nabla u\cdot\vec{\eta}+\gamma(u)\Div\vec{\eta}
    \right]\,\mathrm{d}x\,\mathrm{d}t\\
    &=\frac{1}{6}\int_0^T\int_{\Omega_{\mathrm{in}}(\varphi,t)}\Div(\gamma(u)\vec{\eta})\,\mathrm{d}x\,\mathrm{d}t\\
    &=\frac{1}{6}\int_0^T\int_{\partial^\ast\Omega_{\mathrm{in}}(\varphi,t)}\gamma(u)\vec{\eta}\cdot\vec{\nu}\,\mathrm{d}\mathcal{H}^{n-1}\,\mathrm{d}t\\
    &=\frac{1}{6}\int_0^T\int_\Omega\gamma(u)\vec{\phi}\cdot\vec{\nu}\left(
        \frac{\mathrm{d}\mathcal{H}^{n-1}|_{\partial^\ast\Omega_{\mathrm{in}}(\varphi,t)}}{\mathrm{d}\mu_t}
    \right)\,\mathrm{d}\mu_t\,\mathrm{d}t.
\end{align}
Here, let us show that the following equality holds:
\begin{align}
    \int_0^T\int_{\Omega_{\mathrm{in}}(\varphi,t)}\Div(\gamma(u)\vec{\eta})\,\mathrm{d}x\,\mathrm{d}t
    =\int_0^T\int_{\partial^\ast\Omega_{\mathrm{in}}(\varphi,t)}\gamma(u)\vec{\eta}\cdot\vec{\nu}\,\mathrm{d}\mathcal{H}^{n-1}\,\mathrm{d}t.
    \label{eq:divergence_thm_nonsmooth}
\end{align}
Since $u\in H^1(\Omega\times(0,T))$, we see that
\begin{align}
    \int_0^T\int_\Omega|u(x,t)|^2\,\mathrm{d}x\,\mathrm{d}t,\ 
    \int_0^T\int_\Omega|\nabla u(x,t)|^2\,\mathrm{d}x\,\mathrm{d}t<\infty.
\end{align}
Therefore, 
\begin{align}
    \int_\Omega|u(x,t)|^2\,\mathrm{d}x,\ 
    \int_\Omega|\nabla u(x,t)|^2\,\mathrm{d}x<\infty
    \qquad\text{for a.e. } t\ge0;
\end{align}
that is, $u(\cdot,t)\in H^1(\Omega)$ for a.e. $t\ge0$.
Choose $t\ge0$ such that $u(\cdot,t)\in H^1(\Omega)$ and take an approximate sequence $\{u^j\}_{j=1}^\infty\subset C_0^\infty(\Omega)$; that is, $u^j\to u(\cdot,t)$ in $H^1(\Omega)$.
Since each $u^j$ is smooth, it holds that
\begin{align}
    \int_{\Omega_{\mathrm{in}(\varphi,t)}}\Div(\gamma(u^j)\vec{\eta})\,\mathrm{d}x=\int_{\partial^\ast\Omega_{\mathrm{in}}(\varphi,t)}\gamma(u^j)\vec{\eta}\cdot\vec{\nu}\,\mathrm{d}\mathcal{H}^{n-1}.
    \label{eq:divergence_thm_smooth}
\end{align}
It is clear that
\begin{align}
    \int_{\Omega_{\mathrm{in}}(\varphi,t)}\Div(\gamma(u^j)\vec{\eta})\,\mathrm{d}x
    \to\int_{\Omega_{\mathrm{in}}(\varphi,t)}\Div(\gamma(u)\vec{\eta})\,\mathrm{d}x\qquad\text{as } j\to\infty.
\end{align}
Therefore, we consider the right-hand side of \eqref{eq:divergence_thm_smooth} and show that it is a Cauchy sequence in $\mathbb{R}$.
Choose a function $\zeta\in C_0^\infty(\mathbb{R}^n)$ such that $\zeta=1$ on $\Omega=[0,1)^n$.
Suppose for a moment that $D(t)$, defined by
\begin{align}
    D(t)\coloneqq\sup_{\substack{x\in\Omega\\r\in(0,1)}}\frac{\mathcal{H}^{n-1}(\partial^\ast\Omega_{\mathrm{in}}(\varphi,t)\cap B_r(x))}{r^{n-1}},
\end{align}
is finite for a.e. $t>0$.
Then, it follows from \cite[Lemma 4.9.1]{ziemer1989weakly} that
\begin{align}
    &\left|\int_{\partial^\ast\Omega_{\mathrm{in}}(\varphi,t)}\gamma(u^k)\vec{\eta}\cdot\vec{\nu}\,\mathrm{d}\mathcal{H}^{n-1}-\int_{\partial^\ast\Omega_{\mathrm{in}}(\varphi,t)}\gamma(u^j)\vec{\eta}\cdot\vec{\nu}\,\mathrm{d}\mathcal{H}^{n-1}\right|\\
    &\le C\int_{\partial^\ast\Omega_{\mathrm{in}}(\varphi,t)}|\gamma(u^j)-\gamma(u^k)|\,\mathrm{d}\mathcal{H}^{n-1}
    \le C\int_{\partial^\ast\Omega_{\mathrm{in}}(\varphi,t)}|\zeta(\gamma(u^k)-\gamma(u^j))|\,\mathrm{d}\mathcal{H}^{n-1}\\
    &\le CD(t)\int_{\mathbb{R}^n}|\nabla(\zeta(\gamma(u^k)-\gamma(u^j)))|\,\mathrm{d}x.
\end{align}
Since $\{u^j\}_{j=1}^\infty$ converges to $u(\cdot,t)$ in $H^1(\Omega)$, it is a Cauchy sequence in $H^1(\Omega)$, and
\begin{align}
    &\int_{\mathbb{R}^n}|\nabla(\zeta(\gamma(u^k)-\gamma(u^j)))|\,\mathrm{d}x\\
    &\le\int_{\mathbb{R}^n}|\zeta||\nabla(\gamma(u^k)-\gamma(u^j))|\,\mathrm{d}x+\int_{\mathbb{R}^n}|\nabla\zeta|\cdot|\gamma(u^k)-\gamma(u^j)|\,\mathrm{d}x
    \to0
\end{align}
as $j,k\to\infty$.
Therefore, the right-hand side of \eqref{eq:divergence_thm_smooth} is a Cauchy sequence in $\mathbb{R}$.
Note that
the trace operator 
$T: W^{1,1} (\Omega) \to L^1 (\partial ^\ast \Omega_{\mathrm{in}}(\varphi,t)) $ is well-defined since one can show $\| g^1 - g^2 \|_{L^1 (\partial ^\ast \Omega_{\mathrm{in}}(\varphi,t))} \leq C D(t)
\|g^1 -g^2 \|_{W^{1,1} (\Omega)}$ for any $g^1,g^2 \in W^{1,1} (\Omega)$ similarly.
Hence we see that
\begin{align}
    \int_{\partial^\ast\Omega_{\mathrm{in}}(\varphi,t)}\gamma(u^j)\vec{\eta}\cdot\vec{\nu}\,\mathrm{d}\mathcal{H}^{n-1}\to\int_{\partial^\ast\Omega_{\mathrm{in}}(\varphi,t)}\gamma(u)\vec{\eta}\cdot\vec{\nu}\,\mathrm{d}\mathcal{H}^{n-1}.
\end{align}
This completes the proof of \eqref{eq:divergence_thm_nonsmooth}.

We show that $D(t)$ is finite for a.e. $t>0$.
To this end, we first prove that there exists a subsequence of $\{\varepsilon_i\}_{i=1}^\infty$, denoted by the same symbol, such that
\begin{align}
    \sup_{i\in\mathbb{N}}\left\{
        \int_\Omega\left(\frac{\varepsilon_i|\nabla\varphi^{\varepsilon_i}|^2}{2}+\frac{W(\varphi^{\varepsilon_i})}{2\varepsilon_i}\right)^2\,\mathrm{d}x+\int_\Omega\varepsilon_i\left(\triangle\varphi^{\varepsilon_i}-\frac{W'(\varphi^{\varepsilon_i})}{2\varepsilon_i}\right)^2\,\mathrm{d}x
    \right\}<\infty
    \label{eq:sup_finite}
\end{align}
for a.e. $t>0$.
Since Lemma \ref{lem:estimate}(I) readily shows that
\begin{align}
    \sup_{i\in\mathbb{N}}\int_\Omega\left(\frac{\varepsilon_i|\nabla\varphi^{\varepsilon_i}|^2}{2}+\frac{W(\varphi^{\varepsilon_i})}{2\varepsilon_i}\right)^2\,\mathrm{d}x<\infty,
\end{align}
we focus on the second term.
It follows from Lemma \ref{lem:estimate}(II) that
\begin{align}
    \sup_{i\in\mathbb{N}}\int_0^T\int_\Omega\left(\varepsilon_i\triangle\varphi^{\varepsilon_i}-\frac{W'(\varphi^{\varepsilon_i})}{2\varepsilon_i}\right)^2\,\mathrm{d}x\,\mathrm{d}t<\infty
\end{align}
holds.
Applying the Fatou's lemma, we have
\begin{align}
    &\sup_{i\in\mathbb{N}}\int_0^T\int_\Omega\varepsilon_i\left(\triangle\varphi^{\varepsilon_i}-\frac{W'(\varphi^{\varepsilon_i})}{2\varepsilon_i}\right)^2\,\mathrm{d}x\,\mathrm{d}t\\
    &\ge\liminf_{i\to\infty}\int_0^T\int_\Omega\varepsilon_i\left(\triangle\varphi^{\varepsilon_i}-\frac{W'(\varphi^{\varepsilon_i})}{2\varepsilon_i}\right)^2\,\mathrm{d}x\,\mathrm{d}t\\
    &\ge\int_0^T\left(\liminf_{i\to\infty}\int_\Omega\varepsilon_i\left(\triangle\varphi^{\varepsilon_i}-\frac{W'(\varphi^{\varepsilon_i})}{2\varepsilon_i}\right)^2\,\mathrm{d}x\right)\,\mathrm{d}t.
\end{align}
Therefore,
\begin{align}
    \liminf_{i\to\infty}\int_\Omega\varepsilon_i\left(\triangle\varphi^{\varepsilon_i}-\frac{W'(\varphi^{\varepsilon_i})}{2\varepsilon_i}\right)^2\,\mathrm{d}x
\end{align}
is finite for a.e. $t>0$.
Thus, taking a subsequence of $\{\varepsilon_i\}_{i=1}^\infty$ if necessary, \eqref{eq:sup_finite} holds.
We next prove that
\begin{align}
    D(t)\le\frac{1}{3}\sup_{\substack{x\in\Omega\\r\in(0,1)}}\frac{\mu_t(B_{2r}(x))}{r^{n-1}}
\end{align}
holds.
Note that by applying the Young's inequality, we have
\begin{align}
    |\nabla G(\varphi^\varepsilon)|
    &=|G'(\varphi^\varepsilon)\nabla\varphi^\varepsilon|
    =\sqrt{2W(\varphi^\varepsilon)}|\nabla\varphi^\varepsilon|
    \le\frac{\varepsilon|\nabla\varphi^\varepsilon|^2}{2}+\frac{2W(\varphi^\varepsilon)}{2\varepsilon}\\
    &\le2\left(\frac{\varepsilon|\nabla\varphi^\varepsilon|^2}{2}+\frac{W(\varphi^\varepsilon)}{2\varepsilon}\right).
\end{align}
Using lower semi-continuity, we obtain
\begin{align}
    \mathcal{H}^{n-1}(\partial^\ast\Omega_{\mathrm{in}}(\varphi,t)\cap B_r(x))
    &=\int_\Omega\chi_{B_r(x)}(y)\,\mathrm{d}\|\nabla\chi_{\Omega_{\mathrm{in}}(\varphi,t)}\|(y)\\
    &\le\liminf_{i\to\infty}\int_\Omega\chi_{B_r(x)}(y)\frac{1}{6}|\nabla G(\varphi^{\varepsilon_i}(y,t))|\,\mathrm{d}y\\
    &\le\frac{1}{3}\liminf_{i\to\infty}\int_\Omega\chi_{B_r(x)}(y)\,\mathrm{d}\mu_t^{\varepsilon_i}\\
    &\le\frac{1}{3}\mu_t(B_{2r}(x)).
\end{align}
Since \eqref{eq:sup_finite} holds, using \cite[Proposition 4.5]{roger2006modified}, we see that $D(t)$ is finite for a.e. $t>0$.

Summarizing the above calculations, we conclude that
\begin{equation}
    \vec{v}
    =\vec{h}-\frac{1}{6}\tilde{S}[\varphi]\left(\frac{\mathrm{d}\mathcal{H}^{n-1}|_{\partial^\ast\Omega_{\mathrm{in}}(\varphi,t)}}{\mathrm{d}\mu_t}\right)\vec{\nu}+\frac{1}{6}\gamma(u)\left(
        \frac{\mathrm{d}\mathcal{H}^{n-1}|_{\partial^\ast\Omega_{\mathrm{in}}(\varphi,t)}}{\mathrm{d}\mu_t}
    \right)\vec{\nu}
\end{equation}
in the sense of distribution.
Using the relation~\eqref{eq:RD-der_rel}, we obtain
\begin{equation}
    \vec{v}
    =\vec{h}-\sqrt{2}\tilde{S}[\varphi]\left(\frac{\mathrm{d}\mathcal{H}^{n-1}|_{\partial^\ast\Omega_{\mathrm{in}}(\varphi,t)}}{\mathrm{d}\nu_t}\right)\vec{\nu}+\sqrt{2}\gamma(u)\left(
        \frac{\mathrm{d}\mathcal{H}^{n-1}|_{\partial^\ast\Omega_{\mathrm{in}}(\varphi,t)}}{\mathrm{d}\nu_t}
    \right)\vec{\nu}.
\end{equation}
\end{proof}

\section{Concluding remarks}
In this study, we can derive the following singular limit equations for \eqref{eq:target_problem},  although they are weak solutions:
\begin{equation}
\begin{dcases*}
\vec{v} =\vec{h}-\sqrt{2} \tilde{S}[\varphi]\left(\frac{\mathrm{d}\mathcal{H}^{n-1}|_{\partial^\ast\Omega_{\mathrm{in}}(\varphi,t)}}{\mathrm{d}\nu_t}\right)\vec{\nu} + \sqrt{2} \gamma(u)\left(
        \frac{\mathrm{d}\mathcal{H}^{n-1}|_{\partial^\ast\Omega_{\mathrm{in}}(\varphi,t)}}{\mathrm{d}\nu_t}
    \right)\vec{\nu}, \\
\frac{\partial u}{\partial t}=\triangle u-ku+\varphi. 
\end{dcases*}
\label{eq:u_limit_formal}
\end{equation}
By the way, the singular limit equations of \eqref{eq:target_problem} can be derived formally as follows:
\begin{equation}\label{eq1-12}
 \left\{ \begin{aligned}
   &\vec{v} = \vec{h} + \sqrt{2} 
   \left ( - \tilde{S}[\varphi] + \gamma(u)  \right )\vec{\nu}, \\[1mm]
   & \displaystyle \frac{\partial u}{\partial t} = \triangle u -  k u + \varphi.
\end{aligned} \right.  
\end{equation} 
See \cite{nagayama2023reaction} for formal derivation methods.
Here, we assume that the interface is a single layer in \eqref{eq:u_limit_formal},
\begin{equation}
\frac{\mathrm{d}\mathcal{H}^{n-1}|_{\partial^\ast\Omega_{\mathrm{in}}(\varphi,t)}}{\mathrm{d}\nu_t} = 1.
\end{equation}
Then, \eqref{eq:u_limit_formal} is found to be consistent with the singular limit equations \eqref{eq1-12}. 
Therefore, if we can prove that the weak-strong uniqueness principle \cite{laux2024weak} for \eqref{eq:u_limit_formal}, we can show the mathematical justification for the derivation of the singular limit equations \eqref{eq1-12}.
The above proof is future work.

\subsection*{Acknowledgements}

This work was prtially supported by JSPS KAKENHI Grant Numbers JP22K03425, JP22K18677, JP23H00085, JP23H00086, JP23H04936, JP23K03180, JP23K20808, JP24K00531, JP25K00918, and the Cooperative Research Program of ``Network Joint Research Center for Materials and Devices'' (No. 20251035). 
This work was supported by the Research Institute for Mathematical Sciences, an International Joint Usage/Research Center located
in Kyoto University.

\bibliographystyle{siam}
\bibliography{arXiv}

\begin{thebibliography}{10}

\bibitem{AN2019}
{\sc M.~Akiyama, N.~Nonomura, A.~Tero, and R.~Kobayashi}, {\em Numerical study on spindle positioning using phase field method}, Physical Biology, 16 (2018), p.~016005.

\bibitem{AC1979}
{\sc S.~M. Allen and J.~W. Cahn}, {\em A microscopic theory for antiphase boundary motion and its application to antiphase domain coarsening}, Acta Metallurgica, 27 (1979), pp.~1085--1095.

\bibitem{ambrosio2000functions}
{\sc L.~Ambrosio, N.~Fusco, and D.~Pallara}, {\em Functions of bounded variation and free discontinuity problems}, The Clarendon Press, Oxford University Press, New York, 2000.

\bibitem{AHS2021}
{\sc T.~Q. Ansari, H.~Huang, and S.-Q. Shi}, {\em Phase field modeling for the morphological and microstructural evolution of metallic materials under environmental attack}, npj Comput. Mater., 7 (2021), p.~143.

\bibitem{ecker2004regularity}
{\sc K.~Ecker}, {\em Regularity Theory for Mean Curvature Flow}, Birkh\"auser, 2004.

\bibitem{GG1992}
{\sc F.~Graner and J.~A. Glazier}, {\em Simulation of biological cell sorting using a two-dimensional extended potts model}, Phys. Rev. Lett., 69 (1992), pp.~2013--2016.

\bibitem{HTN2004}
{\sc H.~Honda, M.~Tanemura, and T.~Nagai}, {\em A three-dimensional vertex dynamics cell model of space-filling polyhedra simulating cell behavior in a cell aggregate}, J. Theoret. Biol., 226 (2004), pp.~439--453.

\bibitem{hutchinson1986second}
{\sc J.~E. Hutchinson}, {\em Second fundamental form for varifolds and the existence of surfaces minimising curvature}, Indiana Univ. Math. J., 35 (1986), pp.~45--71.

\bibitem{IKN2014}
{\sc K.~Iida, H.~Kitahata, and M.~Nagayama}, {\em Theoretical study on the translation and rotation of an elliptic camphor particle}, Phys. D, 272 (2014), pp.~39--50.

\bibitem{KIN2013}
{\sc H.~Kitahata, K.~Iida, and M.~Nagayama}, {\em Spontaneous motion of an elliptic camphor particle}, Phys. Rev. E, 87 (2013), p.~010901.

\bibitem{K1993}
{\sc R.~Kobayashi}, {\em Modeling and numerical simulations of dendritic crystal growth}, Physica D: Nonlinear Phenomena, 63 (1993), pp.~410--423.

\bibitem{laux2024weak}
{\sc T.~Laux}, {\em Weak-strong uniqueness for volume-preserving mean curvature flow}, Rev. Mat. Iberoam., 40 (2024), pp.~93--110.

\bibitem{mugnai2008allen}
{\sc L.~Mugnai and M.~R\"{o}ger}, {\em The {A}llen-{C}ahn action functional in higher dimensions}, Interfaces Free Bound., 10 (2008), pp.~45--78.

\bibitem{mugnai2011convergence}
\leavevmode\vrule height 2pt depth -1.6pt width 23pt, {\em Convergence of perturbed {A}llen-{C}ahn equations to forced mean curvature flow}, Indiana Univ. Math. J., 60 (2011), pp.~41--75.

\bibitem{NSKY2005}
{\sc K.~Nagai, Y.~Sumino, H.~Kitahata, and K.~Yoshikawa}, {\em Mode selection in the spontaneous motion of an alcohol droplet}, Phys. Rev. E, 71 (2005), p.~065301.

\bibitem{nagayama2023reaction}
{\sc M.~Nagayama, H.~Monobe, K.~Sakakibara, K.-I. Nakamura, Y.~Kobayashi, and H.~Kitahata}, {\em On the reaction–diffusion type modelling of the self-propelled object motion}, Sci.\ Rep., 13 (12633, 10 pp.).

\bibitem{MNDH2004}
{\sc M.~Nagayama, S.~Nakata, Y.~Doi, and Y.~Hayashima}, {\em A theoretical and experimental study on the unidirectional motion of a camphor disk}, Physica D: Nonlinear Phenomena, 194 (2004), pp.~151--165.

\bibitem{PhysRevE.92.022910}
{\sc K.~Nishi, K.~Wakai, T.~Ueda, M.~Yoshii, Y.~S. Ikura, H.~Nishimori, S.~Nakata, and M.~Nagayama}, {\em Bifurcation phenomena of two self-propelled camphor disks on an annular field depending on system length}, Phys. Rev. E, 92 (2015), p.~022910.

\bibitem{N2012}
{\sc M.~Nonomura}, {\em Study on multicellular systems using a phase field model}, PLOS ONE, 7 (2012), pp.~1--9.

\bibitem{OKS2001}
{\sc M.~Ode, S.~G. Kim, and T.~Suzuki}, {\em Recent advances in the phase-field model for solidification}, ISIJ International, 41 (2001), pp.~1076--1082.

\bibitem{roger2006modified}
{\sc M.~R\"oger and R.~Sch\"atzle}, {\em On a modified conjecture of de giorgi}, Math. Z.,  (2006).

\bibitem{SRL2010}
{\sc D.~Shao, W.-J. Rappel, and H.~Levine}, {\em Computational model for cell morphodynamics}, Phys. Rev. Lett., 105 (2010), p.~108104.

\bibitem{stuvard2024existence}
{\sc S.~Stuvard and Y.~Tonegawa}, {\em On the existence of canonical multi-phase {B}rakke flows}, Adv. Calc. Var., 17 (2024), pp.~33--78.

\bibitem{SMHY2005}
{\sc Y.~Sumino, N.~Magome, T.~Hamada, and K.~Yoshikawa}, {\em Self-running droplet: Emergence of regular motion from nonequilibrium noise}, Phys. Rev. Lett., 94 (2005), p.~068301.

\bibitem{takasao2016existence}
{\sc K.~Takasao and Y.~Tonegawa}, {\em Existence and regularity of mean curvature flow with transport term in higher dimensions}, Math. Ann., 364 (2016), pp.~857--935.

\bibitem{tonegawa2019brakke}
{\sc Y.~Tonegawa}, {\em Brakke's mean curvature flow: An introduction}, SpringerBriefs in Mathematics, Springer, Singapore, 2019.

\bibitem{ziemer1989weakly}
{\sc W.~Ziemer}, {\em Weakly differentiable functions}, Springer-Verlag, 1989.

\end{thebibliography}

\end{document}